\documentclass{amsart}
\usepackage{amsmath,amsthm,amsfonts,amssymb,color,graphicx,appendix,ulem}
\usepackage{comment}
\usepackage{hyperref}

\usepackage{tikz,pgfplots}
\pgfplotsset{width=8.5cm,compat=1.12}

\numberwithin{equation}{section}
\theoremstyle{plain}
\newtheorem{theorem}{\sc Theorem}[section]

\newtheorem{corollary}[theorem]{\sc Corollary}
\newtheorem{definition}[theorem]{\sc Definition}
\newtheorem{lemma}[theorem]{\sc Lemma}
\newtheorem{proposition}[theorem]{\sc Proposition}

\theoremstyle{remark}
\newtheorem{remark}[theorem]{\sc Remark}
\newtheorem{example}[theorem]{\sc Example}

\newcommand{\one}{{{\rm 1\mkern-1.5mu}\!{\rm I}}}
\newcommand{\be}{\begin{equation}}
\newcommand{\ee}{\end{equation}}

\newcommand{\sign}{\mathrm{sign}}

\def\cF{\mathcal{F}}
\def\cG{\mathcal{G}}

\def\cP{\mathcal{P}}

\def\aeL{\alpha^{(x_*,h,y)}}
\def\aleft{\overleftarrow\alpha}
\def\aright{\overrightarrow\alpha}

\renewcommand{\epsilon}{\varepsilon}
\renewcommand{\d}{\partial}
\newcommand{\essinf}{\mathrm{ess\,inf\,}}
\newcommand{\esssup}{\mathrm{ess\,sup\,}}

\begin{document}
	
\title[Homogenization of a class of 1-D nonconvex viscous HJ equations]{Homogenization of a class of one-dimensional nonconvex viscous Hamilton-Jacobi equations with random potential}

\author[E.\ Kosygina]{Elena Kosygina}
\address{Elena Kosygina\\ Department of Mathematics\\ Baruch College\\  One Bernard Baruch Way\\ Box B6-230, New York, NY 10010\\ USA}
\email{elena.kosygina@baruch.cuny.edu}
\urladdr{http://www.baruch.cuny.edu/math/elenak/}
\thanks{E.\ Kosygina was partially supported by the Simons Foundation (Award \#523625)}

\author[A.\ Yilmaz]{Atilla Yilmaz}
\address{Atilla Yilmaz\\ Department of Mathematics\\ Ko\c{c} University\\ Sar\i yer, Istanbul 34450, Turkey\\ and Courant Institute\\ 251 Mercer Street\\ New York, NY 10012\\ USA}
\email{yilmaz@cims.nyu.edu}
\urladdr{http://cims.nyu.edu/$\sim$yilmaz/}
\thanks{A.\ Yilmaz was partially supported by the BAGEP Award of the Science Academy, Turkey.}

\author[O.\ Zeitouni]{Ofer Zeitouni}
\address{Ofer Zeitouni\\ Faculty of Mathematics\\ Weizmann Institute \\ POB 26, Rehovot 76100\\ Israel\\and Courant Institute\\ 251 Mercer Street\\ New York, NY 10012\\ USA}
\email{ofer.zeitouni@weizmann.ac.il}
\urladdr{http://wisdom.weizmann.ac.il/$\sim$zeitouni/}
\thanks{O.\ Zeitouni was partially supported by an Israel Science Foundation grant.} 
	
\date{October 8, 2017.}


\begin{abstract}
    We prove the homogenization of a class of one-dimensional viscous Hamilton-Jacobi equations with random Hamiltonians that are nonconvex in the gradient variable.
    Due to the special form of the Hamiltonians, the solutions of these PDEs with linear initial conditions have representations involving exponential expectations of controlled Brownian motion in a random potential. 
    The effective Hamiltonian is the asymptotic rate of growth of these exponential expectations as time goes to infinity and is explicit in terms of the tilted free energy of (uncontrolled) Brownian motion in a random potential. 
    The proof involves large deviations, construction of correctors which lead to exponential martingales, and identification of asymptotically optimal policies.
\end{abstract}

\maketitle

\section{Introduction}\label{sec:intro}

\subsection{Main results}\label{subsec:main}

Let $(\Omega,{\mathcal F},\mathbb{P})$ be a
probability space and $\{T_y:\Omega\to \Omega\}_{y\in\mathbb{R}}$ a
group of measure preserving transformations with
$T_0=\text{Id},\ T_{x+y}=T_xT_y$.  Furthermore, assume that the group
action is ergodic, that is every set which is invariant under all
$T_y,\ y\in\mathbb{R}$, has measure $0$ or $1$.

We are interested in the behavior as $\epsilon\to 0$ of a family of
solutions $u^\epsilon(t,x,\omega)$, $\epsilon>0$, to the Cauchy problem
\begin{align}\label{eq}
	\frac{\d u^\epsilon}{\d t}=\frac{\epsilon}{2}\frac{\d^2 u^\epsilon}{\d x^2}+H_{\beta,c}\left(\frac{\d u^\epsilon}{\d x},T_{x/\epsilon}\,\omega\right)&,\quad (t,x)\in(0,\infty)\times\mathbb{R},\\\label{ic} u^\epsilon\Big|_{t=0}=g(x)&,\quad x\in\mathbb{R},
\end{align}
where $g$ is in $\text{UC}(\mathbb{R})$, the space of uniformly continuous
functions, and the Hamiltonian
\begin{equation*}\label{H}
	H_{\beta,c}:\mathbb{R}\times \Omega\to \mathbb{R},\quad H_{\beta,c}(p,\omega)=\frac12\,p^2-c|p|+\beta V(\omega),
\end{equation*}
depends on two constant parameters $c\ge 0$ and $\beta>0$. An
important feature of this Hamiltonian is that for all $\beta,c>0$ it
is nonconvex (and not level-set convex) in $p$. The random
environment enters $H_{\beta,c}$ and \eqref{eq} through the potential
$V\in L^\infty(\Omega,{\cF},\mathbb{P})$ and its shifts
$V(T_{x/\epsilon}\omega),\ x\in\mathbb{R}$. We shall assume without
further loss of generality that
\begin{equation}\label{ass:wlog}
	\essinf V(\omega)=0\quad\text{and}\quad \esssup V(\omega)=1.
\end{equation}
The parameter $\beta$ is then just the ``magnitude'' of the
potential. We shall also suppose that $\forall \omega\in\Omega$
\begin{equation}\label{ass:vreg}
	x\mapsto V(T_x\omega)\ \text{is in $C_b^1(\mathbb{R})$}.
\end{equation}
Here and throughout, $C^k$ (resp.\ $C_b^k$), $k=1,2$, refers to the set of functions that are $k$ times differentiable with continuous (resp.\ continuous and uniformly bounded) derivatives up to order $k$ inclusively. The above conditions guarantee that the Cauchy problem \eqref{eq}-\eqref{ic} has a unique viscosity solution in $\text{UC}([0,\infty)\times\mathbb{R})$. 
See Section \ref{subsec:eu} for references.

To state our last assumption on $V$ we need the following definition.
\begin{definition}
	\label{def:vh}
	For any $\omega\in\Omega$ and $h\in(0,1)$, an interval $I$ is said
	to be an $h$-valley (resp.\ $h$-hill) if $V(T_x\omega) \le h$
	(resp.\ $V(T_x\omega) \ge h$) for every $x\in I$.
\end{definition}
We shall assume that for every $h\in(0,1)$ and $y>0$
\begin{equation}
	\mathbb{P}(\text{$[0,y]$ is an $h$-valley}) > 0\ \ \text{and}\ \ \mathbb{P}(\text{$[0,y]$ is an $h$-hill}) > 0.\label{ass:vh}
\end{equation}
Condition \eqref{ass:vh} ensures (for $\mathbb{P}$-a.e.\ $\omega$ by
the ergodicity assumption) the existence of arbitrarily long intervals
where the potential is uniformly close to its ``extremes''.

We prove the following homogenization result.
\begin{theorem}\label{thm:main}
	Assume that $V$ satisfies \eqref{ass:wlog}, \eqref{ass:vreg} and \eqref{ass:vh}.  For every $\theta\in\mathbb{R}$, let $u^\epsilon_\theta$
	be the unique viscosity solution of the Cauchy problem \eqref{eq}-\eqref{ic} in $\text{UC}([0,\infty)\times\mathbb{R})$
	with $g(x)=\theta x$. Then
	\[\mathbb{P}\left(\forall R>0\ \forall T>0\ \lim_{\epsilon\to
		0}\max_{|x|\le R}\max_{t\le
		T}|u^\epsilon_\theta(t,x,\omega)-\overline{H}_{\beta,c}(\theta)t-\theta
	x|=0\right)=1, \]
	where a continuous function $\overline{H}_{\beta,c}$, the effective
	Hamiltonian, is given explicitly in terms of the (non-explicit)
	tilted free energy of a Brownian motion in the potential $\beta V$
	(see \eqref{eq:tiltedfe}, \eqref{eq:weaklimit} and \eqref{eq:stronglimit}).
\end{theorem}

\begin{example}\label{ex:levy}
	Let $f\ge0$ be a $C_b^1(\mathbb{R})$ function with compact support such that $\int_{-\infty}^{\infty}f(x)dx = 1$. Define \[\omega_x = \int_{-\infty}^\infty f(x-y)g(L_y - L_{y-1})dy\] where $L$ is a standard two-sided Poisson or Wiener process with $L_0 = 0$ and $g(a) = (a\vee 0)\wedge 1$. Set $\omega = (\omega_x)_{x\in\mathbb{R}}$ and denote the induced probability space by $(\Omega,\mathcal{F},\mathbb{P})$. Define $T_x$ naturally by $(T_x\omega)_y = \omega_{x+y}$ and let $V(\omega) = \omega_0$.
\end{example}

We leave it to the reader to check that Example \ref{ex:levy} falls
within our model and satisfies conditions
\eqref{ass:wlog}-\eqref{ass:vh}. The potential in this example has a finite range of dependence.
However, our assumptions in general do not imply that the potential is even
weakly mixing. In the discrete setting, this was shown in \cite[Example 1.3]{YZ17} and the argument carries over easily
to the continuous setting.

Note that the function
$\overline{u}_\theta(t,x)=\overline{H}_{\beta,c}(\theta)t+\theta x$
satisfies the equation
\begin{equation}
	\label{eqh}
	\frac{\d \overline{u}}{\d t}=\overline{H}_{\beta,c}\left(\frac{\d \overline{u}}{\d x}\right),\quad (t,x)\in(0,\infty)\times\mathbb{R},
\end{equation}
and the initial condition $\overline{u}(0,x)=\theta x$. 
It is the unique viscosity solution of this Cauchy problem in $\text{UC}([0,\infty)\times\mathbb{R})$, see Section \ref{subsec:eu}.
Moreover, $\overline{H}_{\beta,c}(\theta) = \overline{u}_\theta(1,0) = \lim_{\epsilon\to 0}u^\epsilon_\theta(1,0,\omega)$.
Using general results from \cite{DK17} we deduce the
following corollary.
\begin{corollary}\label{cor:main}
	Assume that $V$ satisfies the conditions in Theorem~\ref{thm:main}. For
	every $g\in\text{UC}(\mathbb{R})$, let $u^\epsilon_g(t,x,\omega)$
	be the unique viscosity solution of \eqref{eq} in
	$\text{UC}([0,\infty)\times\mathbb{R})$ with the initial condition
	$g$. Then
	\[\mathbb{P}\left(\forall g\in\text{UC}(\mathbb{R})\ \forall R>0\
	\forall T>0\ \lim_{\epsilon\to 0}\max_{|x|\le R}\max_{t\le
		T}|u^\epsilon_g(t,x,\omega)-\overline{u}_g(t,x)|=0\right)=1,\]
	where $\overline{u}_g$ is the unique viscosity solution of
	\eqref{eqh} in $\text{UC}([0,\infty)\times\mathbb{R})$ with the initial condition $g$.
\end{corollary}
Thus, we obtain a full homogenization result for a new class of
viscous Hamilton-Jacobi equations with nonconvex (when $c> 0$)
Hamiltonians in dimension 1.

The solution of \eqref{eq}-\eqref{ic} rewritten as a terminal value
problem is also known to characterize the value of a two-player,
zero-sum stochastic differential game (see, for instance,
\cite{FlS89}). But due to the special form of \eqref{eq}, our
homogenization problem admits a simple and useful control
interpretation, where, roughly speaking, the role of one of the players is
implicitly assumed by the diffusion in the random environment. More
precisely, let $(X_t)_{t\ge 0}$ be a standard Brownian motion (BM) that is independent of the environment. We denote by
$(\mathcal{G}_t)_{t\ge 0}$ the natural filtration and by $P_0$ (resp.\ $E_0$) the probability (resp.\ expectation) corresponding to this BM when $X_0=0$. By the Hopf-Cole
transformation and scaling, i.e., setting
$u^\epsilon_\theta(t,x,\omega)=\epsilon\log
v_\theta(t/\epsilon,x/\epsilon,\omega)$,
we get that $v_\theta(t,x,\omega)$ satisfies the
$\epsilon$-independent equation
\[\frac{\d v_\theta}{\d t}=\frac12\frac{\d^2 v_\theta}{\d
	x^2}-c\left|\frac{\d v_\theta}{\d x}\right|+\beta
V(T_x\omega)v_\theta,\quad (t,x)\in(0,\infty)\times\mathbb{R},\]
and the initial condition $v_\theta(0,x,\omega)=e^{\theta x}$. Taking for
simplicity $t=1$ and $x=0$ and using the control representation for
$v_\theta$ (see \eqref{eq:protoconrep}-\eqref{eq:controlrep}), the limiting behavior of $u^\epsilon_\theta(1,0,\omega)$
as $\epsilon\to 0$ boils down to showing the existence of the limit
(with $S=1/\epsilon$)
\begin{equation}\label{cr}
	\overline{H}_{\beta,c}(\theta)=\lim_{\epsilon\to 0}u^\epsilon_\theta(1,0,\omega)=\lim_{S\to\infty}\inf_{\alpha\in {\cP}_c}\frac{1}{S}\log E_0\left[e^{\beta \int_0^S V(T_{X^\alpha_s}\omega)\,ds+\theta X^\alpha_S}\right],
\end{equation}
where
\[{\cP}_c=\{\alpha=(\alpha_s)_{s\ge 0}: \ \alpha\ \text{is $[-c,c]$-valued and $\mathcal {G}_s$-progressively measurable}\}\]
is the set of admissible controls, and $(X^\alpha_s)_{s\ge 0}$ is
defined by
\begin{equation}\label{eq:SDEmiz}
X^\alpha_s=\int_0^s\alpha_r\,dr+X_s.
\end{equation}
The control interpretation \eqref{cr} indicates that, contrary to the
case $c=0$ for which the existence of the limit can be easily shown
by subadditivity arguments, allowing $c>0$ destroys
subadditivity. Hence, even the existence of the limit in \eqref{cr}
becomes a nontrivial statement.  We are able not only to
show that the limit in \eqref{cr} exists and get a semi-explicit
expression for it but also to provide asymptotically
optimal controls (see Section \ref{subsec:EffH}).

Recall that the original Hamiltonian
$H_{\beta,c}(p,\omega)$ is nonconvex in $p$ for all $\beta,c>0$. We
show that in our setting the convexity/nonconvexity of the effective
Hamiltonian depends only on the magnitude of the potential and the size of the
control. More precisely, $\overline{H}_{\beta,c}$ is
convex if and only if $\beta \ge c^2/2$ (see Figure \ref{matrixfigure}).
The ``convexification'' of the effective Hamiltonian has
been previously observed for the first order Hamilton-Jacobi equations
in \cite{ATY15}, \cite{ATY16}, \cite{QTY17+}.

\subsection{Broader context}

We shall refer to the equation of the form
\begin{equation}
	\label{geq}
	\frac{\d u^\epsilon}{\d
		t}=\frac{\epsilon}{2}\text{tr\,}(A(T_{x/\epsilon}\omega)D^2u^\epsilon)+H(Du^\epsilon,
	T_{x/\epsilon}\omega),\quad (t,x)\in(0,\infty)\times\mathbb{R}^d,
\end{equation}
as a viscous Hamilton-Jacobi equation if the symmetric positive semi-definite
matrix $A\not\equiv 0$ and as an inviscid Hamilton-Jacobi equation (or
simply Hamilton-Jacobi equation) if $A\equiv 0$.

By change of variables we can always write
$u^\epsilon(t,x,\omega)=\epsilon u(t/\epsilon,x/\epsilon,\omega)$
where $u(t,x,\omega)$ solves \eqref{geq} for $\epsilon=1$. Thus,
we are interested in the existence of a scaling limit under the
  hyperbolic scaling of time and space. Note that linear initial conditions are invariant under this scaling.
  We shall say that \eqref{geq} with initial condition $g(x)$ homogenizes if with probability
  one $u^\epsilon(t,x,\omega)$ converges locally
    uniformly in $t$ and $x$ to the solution $\overline u(t,x)$ of a
    deterministic PDE with the same initial condition. If the
  convergence is only in probability then we shall say that the
  problem homogenizes in probability.

  It has been shown that if the Hamiltonian $H(p,\omega)$ is convex in
  the momentum variables ($p\in\mathbb{R}^d$), then
  homogenization holds for very general viscous and inviscid
  Hamilton-Jacobi equations in all dimensions. The literature on the
  subject is vast. We shall focus primarily on the viscous case and
  refer the reader to \cite{LS05}, \cite{KRV06}, \cite{AS12},
  \cite{AT14}, \cite{AC15}, and the references therein.

If the Hamiltonian $H(p,x)$ is 1-periodic in each of the spatial
variables, then there is a general method of proving homogenization due
to \cite{LPV} (see \cite{E92} for an extension to general first
and second order fully nonlinear PDEs). The method is based on the
construction of correctors. Correctors are sublinear (at infinity)
solutions to a certain family of nonlinear eigenvalue problems, see
\eqref{eq:auxODE} for our case. The applicability of the method does not depend
on the convexity of the Hamiltonian but rather on the coercivity of the
Hamiltonian in $p$ and the compactness of the space ($x$ changes on a
torus). The method of correctors originated in the study of linear
partial differential equations with periodic coefficients, and we
refer the interested reader to the monographs \cite{BLP78}, \cite{JKO94}.

In the more general stochastic setting, \cite[Theorem 4.1]{RT00}
asserted that if sublinear correctors exist for each vector in
$\mathbb{R}^d$ then the (inviscid) stochastic problem
homogenizes. However, soon it was shown (\cite{LS03}) that in the
stochastic case correctors need not exist in general. Thus,
other methods were developed. One of them was introduced already
in \cite{LPV} as an alternative method for the (inviscid) periodic
problem with a convex Hamiltonian. Convexity played an important
role. First of all, it allowed one to use the control representation
of solutions. Furthermore, the convexity assumption implied
the subadditivity of certain solution-defining quantities, and the
homogenization result could be obtained from a subadditive ergodic
theorem.

The first two papers which addressed the stochastic homogenization of
viscous Hamilton-Jacobi equations, \cite{LS05} and \cite{KRV06}, used
variational representations of the solutions. In an attempt to steer away
from representation formulas and following some ideas in \cite{Sz94}
and \cite{LS10}, the paper \cite{AS12} introduced a method based on
the so-called metric problem. The convexity assumption gives a subadditivity
property to solutions of the metric problem and, thus, still
essentially restricts the approach to convex Hamiltonians (level-set
convex in the inviscid case, \cite{AS13}). This approach was further
developed in \cite{AT14} and \cite{AC15} where some assumptions are relaxed
and a rate of convergence is obtained.

For quite some time it was not clear whether the convexity assumption
can be disposed of in the stochastic setting. Several classes of
examples of nonconvex Hamiltonians for which homogenization holds
were recently constructed: \cite{ATY15}, \cite{ATY16}, \cite{G16},
\cite{FeS16+}, \cite{QTY17+} for inviscid equations and \cite{AC17},
\cite{DK17} for viscous equations. On the other hand, the work
\cite{Z17} has demonstrated that homogenization could fail for
nonconvex Hamiltonians in the general stationary and ergodic setting
for dimensions $d\ge2$.
(See also \cite{FeS16+}.) These results indicate that the
homogenization or non-homogenization for nonconvex Hamiltonians
depends significantly on the interplay between the nonconvexity and
the random environment, and that one cannot expect a comprehensive
solution. The case of viscous equations is particularly difficult,
since the viscosity term adds yet another randomness, encoded in the
diffusion. The critical scaling at which the diffusion enters the
equation brings into play the large deviations for this diffusion and
makes the analysis even more challenging. 

In spite of the negative results of \cite{LS03} on the existence of
correctors for \eqref{geq}, the quest for them has never ended (see,
for example, \cite{DS09}, \cite{DS12}, \cite{AC17}). In particular,
the authors of \cite{CS17} have shown under quite general conditions
that if the equation \eqref{geq} homogenizes in probability to an
effective equation with some continuous and coercive Hamiltonian
$\overline{H}(p)$ then correctors exist for every $p$ that is an
extreme point of the convex hull of the sub-level set
$\{q\in\mathbb{R}^d:\,\overline{H}(q)\le \overline{H}(p)\}$. In
discrete multidimensional settings such as first passage percolation, random walks in
random environments and directed polymers, the existence of correctors
(and closely related Busemann functions) also received a lot of
attention (see, e.g., \cite{Y11}, \cite{DH14}, \cite{K16},
\cite{GRAS16}, \cite{BL16}).

\subsection{Motivation, method and outline}

Given the above developments and the complexity of the general
homogenization problem with a nonconvex
Hamiltonian, one can start with a more modest goal and look first at
some model examples of viscous Hamilton-Jacobi equations in dimension
1. For the inviscid case in dimension 1 there are already quite general homogenization
results, see \cite{ATY15} and \cite{G16}. It is natural to conjecture
that homogenization in dimension 1 holds under general assumptions in
the viscous case as well. Methods which were used in the inviscid case
are not applicable to the viscous case due to the presence of the
diffusion term.

In this paper we provide a new class of examples in the viscous
one-dimensional case for which homogenization holds. Our examples, in a way,
complement some of those considered earlier in \cite[Theorem
4.10]{DK17}. More precisely, the method of \cite{DK17} is applicable
to Hamiltonians which have one or more ``pinning points'' (i.e.\
values $p^*$ such that $H(p^*,\cdot)\equiv \text{const}$) and are
convex in $p$ in between the pinning points. For example,
$H(p,\omega)=\frac12\,p^2-b(\omega)|p|$, $b(\omega)>0$, is pinned at
$p^*=0$ and is convex in $p$ on $(-\infty,0)$ and $(0,\infty)$. Adding
a non-constant potential breaks the pinning property. Currently it is
not known if \eqref{geq} with
\begin{equation}
	\label{gqc}
	H(p,\omega)=\frac12\,|p|^2-b(\omega)|p|+\beta V(\omega)
\end{equation}
and $A\equiv \text{const}\not\equiv 0$ homogenizes even in dimension 1. Our
results give a positive answer in the case when $b(\omega)\equiv c>0$
and $V$ satisfies \eqref{ass:wlog}-\eqref{ass:vh}.

To prove Theorem~\ref{thm:main} we first consider the convex case
$c=0$ (no-control case for \eqref{cr}) and construct a
function $F_{\beta,\theta}$ for every $\theta$ outside of the
closed interval where the tilted free energy $\Lambda_\beta(\theta)$
(see \eqref{eq:tiltedfe}) attains its minimum value $\beta$, i.e.\
outside of the so-called ``flat piece'' of the effective Hamiltonian
$\overline{H}_{\beta,0}\equiv \Lambda_\beta$.
The function $F_{\beta,\theta}$ is used to introduce the exponential martingale $M_t(\omega)$, see \eqref{eq:martan}; for this reason we call $F_{\beta,\theta}$ a corrector.
The martingale $M_t(\omega)$ immediately yields the desired limit.
For each $c>0$ we build the effective Hamiltonian $\overline{H}_{\beta,c}$ by shifting
and bridging together pieces of $\Lambda_\beta(\cdot\pm c)$ (see
Section~\ref{subsec:EffH} and Figure~\ref{matrixfigure}).
We note that in our setting correctors exist for all $\theta$
outside of the ``flat pieces'' of $\overline{H}_{\beta,c}(\theta)$
and coincide with those constructed in the no-control case for
an appropriately shifted $\theta$.
We use each of these correctors to define an exponential expression (see \eqref{eq:submartan}) which turns out to be (i) a submartingale for arbitrary control policies and (ii) a martingale for specific control policies that we deduce to be asymptotically optimal.

The above approach was first proposed and implemented in \cite{YZ17} in the discrete setting
where the BM in the control problem \eqref{cr} is replaced by a random walk, the analog of Theorem \ref{thm:main} is proved for a viscous Hamilton-Jacobi partial difference equation and the effective Hamiltonian is shown to have the same structure as in this paper. However,
as it is often the case, the arguments in the continuous formulation differ noticeably.
We believe that some of the ideas in \cite{YZ17} and this paper can be extended to
more general settings, for example, to Hamiltonians of the form \eqref{gqc} in one or more dimensions.

We end this introduction with a brief outline of the rest of the paper. Section \ref{sec:nocon} focuses on the no-control case. We obtain a uniform lower bound using the existence of arbitrarily long high hills, construct the aforementioned correctors, use them to give a self-contained proof of the existence of the tilted free energy (Theorem \ref{thm:tiltedfe}) and list some of the properties of the latter (Proposition \ref{prop:list}).
	Section \ref{sec:UB} contains upper bounds for the control problem. We restrict the infimum in \eqref{cr} to bang-bang policies, consider the constant policies $\aleft\equiv -c$ and $\aright\equiv c$ as well as a family of policies $\alpha^{(x_*,h,y)}$ which tries to trap the particle to low valleys. These upper bounds produce the graphs in Figure \ref{matrixfigure}.
Section \ref{sec:LB} provides matching lower bounds. We obtain a uniform lower bound similar to that in the no-control case, use the correctors outside of the flat pieces as mentioned above, give a scaling argument at the elevated flat piece centered at the origin when $\beta<c^2/2$, and thereby prove the existence of the effective Hamiltonian (Theorem \ref{thm:control}). 
	Finally, Section \ref{sec:homogen} wraps up the solution of the homogenization problem. We derive the control representation and put our limit results (Theorems \ref{thm:tiltedfe} and \ref{thm:control}) together with relevant results from \cite{DK17} to prove Theorem \ref{thm:main} and Corollary \ref{cor:main}. Note that the regularity assumption \eqref{ass:vreg} on $V$ is used only in Section \ref{sec:homogen} and it is replaced by the weaker continuity assumption \eqref{ass:vhol} in Sections \ref{sec:nocon}, \ref{sec:UB} and \ref{sec:LB}.

\section{No control}\label{sec:nocon}

We start our analysis with the special case of no control, $c=0$, where the limit on the RHS of \eqref{cr} simplifies to 
\begin{equation}\label{eq:tiltedfe}
\overline H_{\beta,0}(\theta) = \Lambda_\beta(\theta) = \lim_{t\to\infty}\frac1{t}\log E_0\left[e^{\beta\int_{0}^tV(T_{X_s}\omega)ds + \theta X_t}\right].
\end{equation}
In this section, we assume that $V$ satisfies \eqref{ass:wlog}, \eqref{ass:vh} and $\forall\omega\in\Omega$
\begin{equation}\label{ass:vhol}
x\mapsto V(T_x\omega)\ \text{is H\"older continuous with some positive exponent.}
\end{equation}
Under these assumptions, we prove that the limit in \eqref{eq:tiltedfe} exists for all $\beta>0$, $\theta\in\mathbb{R}$ and $\mathbb{P}$-a.e.\ $\omega$.
To this end, we define
\begin{equation}
\begin{aligned}\label{eq:sondis}
\Lambda_\beta^L(\theta) &= \liminf_{t\to\infty}\frac1{t}\log E_0\left[e^{\beta\int_{0}^tV(T_{X_s}\omega)ds + \theta X_t}\right]\\
\text{and}\hspace{15mm}  &\hspace{70mm} \\
\Lambda_\beta^U(\theta) &= \limsup_{t\to\infty}\frac1{t}\log E_0\left[e^{\beta\int_{0}^tV(T_{X_s}\omega)ds + \theta X_t}\right].
\end{aligned}
\end{equation}

The quantity $\Lambda_\beta(\theta)$ is referred to as the tilted free energy of BM. Its existence is covered by the works \cite{LS05, KRV06} on the homogenization of viscous Hamilton-Jacobi equations with convex Hamiltonians. We give a short and self-contained proof which relies on the construction of correctors and provides an implicit formula for $\Lambda_\beta(\theta)$. Introducing these correctors is in fact our main motivation here since they will play a key role in our solution of the control problem (i.e.\ showing the existence of the limit on the RHS of \eqref{cr} for $c>0$) in Section \ref{sec:LB}.

\subsection{Uniform lower bound (no control)}

Throughout the paper, we make use of the hitting times
\[\tau_y = \inf\{t\ge0:\,X_t = y\}\quad\text{and}\quad \tau_{\pm y} = \tau_{-y}\wedge\tau_y,\quad y\in\mathbb{R}.\]

\begin{lemma}\label{lem:eig}
	For every $y>0$,
	\[\lim_{t\to\infty}\frac1{t}\log P_0(\tau_{\pm y} > t) = -\frac{\pi^2}{8y^2}.\]
\end{lemma}

\begin{proof}
	This follows immediately from the spectral analysis of the Laplace operator on $[-y,y]$ with Dirichlet boundary conditions (see \cite[Section 5.8]{V07}).
\end{proof}

\begin{lemma}\label{lem:unifLB}
	For every $\beta>0$, $\theta\in\mathbb{R}$ and $\mathbb{P}$-a.e.\ $\omega$, we have $\Lambda_\beta^L(\theta) \ge \beta$ with the notation in \eqref{eq:sondis}.
\end{lemma}

\begin{proof}
	By \eqref{ass:vh} and ergodicity, for every $h\in(0,1)$, $y>0$ and $\mathbb{P}$-a.e.\ $\omega$, there is an $h$-hill of the form $[x^*-y,x^*+y]$ for some $x^*\in\mathbb{R}$ (which depends on $\omega$). Using the strong Markov property and Lemma \ref{lem:eig}, we get
	\[\Lambda_\beta^L(\theta) \ge \liminf_{t\to\infty}\frac1{t}\log E_0\left[e^{\beta\int_{0}^tV(T_{X_s}\omega)ds + \theta X_t}\one_{\{\tau_{x^*} \le 1\}\cap\{|X_s- x^*|<y\ \text{for every}\ s\in[\tau_{x^*},t]\}}\right] \ge \beta h - \frac{\pi^2}{8y^2}.\]
	Finally, we send $h\to1$ and $y\to\infty$.
\end{proof}

\subsection{Correctors}\label{subsec:cor}

For every $\beta>0$, $\lambda\ge\beta$, $\omega\in\Omega$ and $x,y\in\mathbb{R}$, let
\[v_\beta^\lambda(\omega,x;y) = E_x\left[e^{\beta\int_{0}^{\tau_y}V(T_{X_s}\omega)ds - \lambda\tau_y}\right].\]
We make two elementary observations. First,
\begin{equation}\label{eq:mell}
v_\beta^\lambda(\omega,x;z) = v_\beta^\lambda(\omega,x;y)\,v_\beta^\lambda(\omega,y;z)
\end{equation}
by the strong Markov property of BM whenever $x\le y\le z$ or $x\ge y\ge z$. Second, since $V(\cdot)\in[0,1]$, it follows from Lemma \ref{lem:hittime} (below) that
\begin{equation}\label{eq:esti}
\sqrt{2(\lambda - \beta)}|x-y| \le - \log v_\beta^\lambda(\omega,x;y) \le \sqrt{2\lambda}|x-y|.
\end{equation}

\begin{lemma}\label{lem:hittime}
	For every $a \ge 0$ and $x,y\in\mathbb{R}$,
	\[E_x\left[e^{-a\tau_y}\right] = e^{-\sqrt{2a}|x-y|}.\]
\end{lemma}

\begin{proof}
	See \cite[Chapter 2, Proposition 3.7]{RY99}.
\end{proof}

For every $\beta>0$, $\theta\ne 0$, $\lambda\ge\beta$, $\omega\in\Omega$ and $x\in\mathbb{R}$, let
\begin{align}
F_{\beta,\theta}^\lambda(\omega,x) &= \begin{cases}
-\log v_\beta^\lambda(\omega,0;x) - \theta x & \text{if $\theta x \ge 0$,}\\
\quad \log v_\beta^\lambda(\omega,x;0) - \theta x & \text{if $\theta x < 0$.}\end{cases}\label{eq:defF}
\end{align}
It follows from \eqref{eq:mell} that
\begin{equation}\label{eq:datsi}
F_{\beta,\theta}^\lambda(\omega,x) = \log v_\beta^\lambda(\omega,x;z) -\log v_\beta^\lambda(\omega,0;z) - \theta x
\end{equation}
for every $z\in\mathbb{R}$ such that $(\theta x)^+ \le \theta z$. Using this representation, it is easy to check that
\begin{equation}\label{eq:cocycle}
F_{\beta,\theta}^\lambda(\omega,x) + F_{\beta,\theta}^\lambda(T_x\omega,y) = F_{\beta,\theta}^\lambda(\omega,x+y)
\end{equation}
for every $\omega\in\Omega$ and $x,y\in\mathbb{R}$. We refer to this identity as the cocycle property.

In the following lemma and the rest of the paper, we use the notation $(\cdot)' = \frac{\partial}{\partial x}(\cdot)$ and $(\cdot)'' = \frac{\partial^2}{\partial x^2}(\cdot)$. 

\begin{lemma}\label{lem:cell}
	For every $\beta>0$, $\theta\ne 0$, $\lambda\ge\beta$ and $\omega\in\Omega$, the function $x\mapsto F_{\beta,\theta}^\lambda(\omega,x)$ is in $C^2(\mathbb{R})$ and
	\[\frac1{2}\left(F_{\beta,\theta}^\lambda\right)'' + \frac1{2}\left(\theta + \left(F_{\beta,\theta}^\lambda\right)'\right)^2 + \beta V(T_x\omega) = \lambda,\quad x\in\mathbb{R}.\]
\end{lemma}

\begin{proof}
	For every $\beta>0$, $\lambda\ge\beta$, $\omega\in\Omega$ and $y\in\mathbb{R}$, the function $x\mapsto v_\beta^\lambda(\omega,x;y)$ is in $C^2(\mathbb{R}\setminus\{y\})$ and
	\[\frac1{2}\left(v_\beta^\lambda\right)''(\omega,x;y) + \left(\beta V(T_x\omega) - \lambda\right)v_\beta^\lambda(\omega,x;y) = 0,\quad x\in\mathbb{R}\setminus\{y\},\]
	by \eqref{ass:wlog} and \eqref{ass:vhol}. See \cite[Chapter 6, Section 3]{Dyn02}. The desired result follows from taking the logarithm of $v_\beta^\lambda$ and using the representation in \eqref{eq:datsi}.
\end{proof}

So far we have been working with an arbitrary $\lambda\ge\beta$. In Lemma \ref{lem:minzer} below, we identify a particular choice of $\lambda$.

\begin{lemma}\label{lem:dombr}
	If $\theta > 0$ and $\beta\le\lambda<\Lambda_\beta^U(\theta)$, then $\mathbb{E}[F_{\beta,\theta}^\lambda(\cdot,1)] \le 0$.
\end{lemma}

\begin{proof}
	For every $t>0$,
	\begin{align}
		E_0\left[e^{\beta\int_0^tV(T_{X_s}\omega)ds + \theta X_t}\right]e^{ - t\lambda} &= \sum_{k=0}^\infty E_0\left[e^{\int_0^t[\beta V(T_{X_s}\omega) - \lambda]ds  + \theta X_t}\one_{\{\tau_{k} < t \le \tau_{k + 1}\}}\right]\nonumber\\
		&\le \sum_{k=0}^\infty E_0\left[e^{\int_0^t[\beta V(T_{X_s}\omega) - \lambda]ds + \theta(k+1)}\one_{\{\tau_{k} < t \le \tau_{k + 1}\}}\right]\nonumber\\
		&\le \sum_{k=0}^\infty e^{\theta}E_0\left[e^{\int_0^{\tau_k}[\beta V(T_{X_s}\omega) - \lambda]ds + \theta k}\right]\nonumber\\
		&= \sum_{k=0}^\infty e^{\theta}\prod_{j=0}^{k-1}E_j\left[e^{\int_0^{\tau_{j+1}}[\beta V(T_{X_s}\omega) - \lambda]ds}\right]e^\theta\label{eq:tirrik}\\
		&= \sum_{k=0}^\infty e^{\theta + \sum_{j=0}^{k-1}[\log v_\beta^\lambda(T_j\omega,0;1) + \theta]} = \sum_{k=0}^\infty e^{\theta - \sum_{j=0}^{k-1}F_{\beta,\theta}^\lambda(T_j\omega,1)}.\label{eq-cancun2}
	\end{align}
	The equality in \eqref{eq:tirrik} follows from the strong Markov property of BM. We take $\limsup$ as $t\to\infty$, use $\Lambda_\beta^U(\theta) - \lambda > 0$ and deduce that the RHS of \eqref{eq-cancun2} is infinite. This implies that $\mathbb{E}[F_{\beta,\theta}^\lambda(\cdot,1)] \le 0$, since otherwise the RHS of \eqref{eq-cancun2} would be finite by the Birkhoff ergodic theorem.
\end{proof}

\begin{lemma}\label{lem:minzer}
	If $\beta<\Lambda_\beta^U(\theta)$, then there exists a $\lambda_o = \lambda_o(\beta,\theta) \ge \Lambda_\beta^U(\theta)$ such that $\mathbb{E}[F_{\beta,\theta}^{\lambda_o}(\cdot,1)] = 0$. 
\end{lemma}

\begin{proof}
	Assume that $\theta>0$. The map $\lambda\mapsto\mathbb{E}[F_{\beta,\theta}^\lambda(\cdot,1)]$ 
	is continuous and strictly increasing for $\lambda\ge\beta$. Moreover,
	\[\lim_{\lambda\to\infty}\mathbb{E}\left[F_{\beta,\theta}^\lambda(\cdot,1)\right] = \infty.\]
	These assertions follow from the $P_0$-a.s.\ positivity of $\tau_1$, the uniform (in $\omega$) bounds in \eqref{eq:esti} and the dominated convergence theorem. Recalling Lemma \ref{lem:dombr}, we deduce the desired result. The $\theta<0$ case is similar. 
\end{proof}

In the rest of this paper, we write $F_{\beta,\theta}(\omega,x) = F_{\beta,\theta}^{\lambda_o}(\omega,x)$ for notational brevity. We give two lemmas which are elementary but of central importance in our analysis both when $c=0$ and $c>0$.

\begin{lemma}\label{lem:simlaz}
	Assume that $\beta<\Lambda_\beta^U(\theta)$. Then, the following bounds hold for every $\omega\in\Omega$ and $x\in\mathbb{R}$.
	\begin{equation}
	\theta + (F_{\beta,\theta})'(\omega,x) \in \begin{cases}
	[\sqrt{2(\lambda_o(\beta,\theta) - \beta)},\sqrt{2\lambda_o(\beta,\theta)}] & \text{if $\theta > 0$,}\\
	[-\sqrt{2\lambda_o(\beta,\theta)},-\sqrt{2(\lambda_o(\beta,\theta) - \beta)}] & \text{if $\theta < 0$.}\end{cases}\nonumber
	\end{equation}
\end{lemma}

\begin{proof}
	For every $\omega\in\Omega$, $x\in\mathbb{R}$ and $y>0$, we use the cocycle property \eqref{eq:cocycle} to write
	\[\frac1{y}\left(F_{\beta,\theta}(\omega,x+y) - F_{\beta,\theta}(\omega,x)\right) = \frac1{y}F_{\beta,\theta}(T_x\omega,y),\]
	recall \eqref{eq:esti}-\eqref{eq:defF}, send $y\to0$ and deduce the desired bounds.
\end{proof}

\begin{lemma}\label{lem:sublin}
	If $\beta<\Lambda_\beta^U(\theta)$, then $\sup\{|F_{\beta,\theta}(\omega,x)|: |x|\le t\} = o(t)$ for $\mathbb{P}$-a.e.\ $\omega$.
\end{lemma}

\begin{proof}
	For $\mathbb{P}$-a.e.\ $\omega$, $F_{\beta,\theta}(\omega,x) = o(|x|)$ by Lemma \ref{lem:minzer}, the cocycle property \eqref{eq:cocycle} and the Birkhoff ergodic theorem. Hence, for every $\epsilon>0$, there exists a $t_0 = t_0(\epsilon)$ such that $|F_{\beta,\theta}(\omega,x)| \le \epsilon |x|$ for $|x| \ge t_0$. From this we deduce that
	\begin{align*}
	t^{-1}\sup\{|F_{\beta,\theta}(\omega,x)|: |x|\le t\} &\le t^{-1}\sup\{|F_{\beta,\theta}(\omega,x)|: |x|\le t_0\}\vee |x|^{-1}\sup\{|F_{\beta,\theta}(\omega,x)|: |x| \ge t_0\}\\
	&\le t^{-1}\sup\{|F_{\beta,\theta}(\omega,x)|: |x|\le t_0\}\vee \epsilon \le \epsilon
	\end{align*}
	for sufficiently large $t > t_0$.
\end{proof}

To summarize, the function $x\mapsto F_{\beta,\theta}(\omega,x)$ is defined whenever $\beta<\Lambda_\beta^U(\theta)$ (see Lemma \ref{lem:minzer}). For $\mathbb{P}$-a.e.\ $\omega$, it is a sublinear (at infinity) solution of
\begin{equation}\label{eq:auxODE}
\frac1{2}\left(F_{\beta,\theta}\right)'' + \frac1{2}\left(\theta + \left(F_{\beta,\theta}\right)'\right)^2 + \beta V(T_x\omega) = \lambda_o(\beta,\theta),\quad x\in\mathbb{R},
\end{equation}
by Lemmas \ref{lem:cell} and \ref{lem:sublin}. We refer to this family of functions as correctors.

\subsection{The tilted free energy}\label{subsec:tfe}

For every $t\ge0$, $\omega\in\Omega$, $\beta>0$ and $\theta\ne 0$ such that $\beta<\Lambda_\beta^U(\theta)$, let
\begin{equation}\label{eq:martan}
M_t(\omega) = M_t(\omega\,|\,\beta,\theta) = e^{\beta\int_{0}^tV(T_{X_s}\omega)ds + \theta X_t + F_{\beta,\theta}(\omega,X_t) - \lambda_o(\beta,\theta) t}.
\end{equation}
Observe that $M_0(\omega) = e^{\theta x + F_{\beta,\theta}(\omega,x)}$ when $X_0 = x$. Since
\begin{equation}\label{eq:lzmol}
0 < M_t(\omega) < e^{\theta X_t + F_{\beta,\theta}(\omega,X_t)} \le e^{\sqrt{2\lambda_o}|X_t|}
\end{equation}
by Lemma \ref{lem:simlaz}, we have $E_0[(M_t(\omega))^2] < \infty$.
Applying It\^{o}'s lemma,
\begin{align*}\frac{d M_t(\omega)}{M_t(\omega)} &= (\theta + (F_{\beta,\theta})'(\omega,X_t))dX_t\\
&\quad + \frac1{2}(F_{\beta,\theta})''(\omega,X_t) dt + \frac1{2}(\theta + (F_{\beta,\theta})'(\omega,X_t))^2 dt + (\beta V(T_{X_t}\omega) - \lambda_o(\beta,\theta))dt\\
&= (\theta + (F_{\beta,\theta})'(\omega,X_t))dX_t.
\end{align*}
The last equality uses \eqref{eq:auxODE}. Therefore, for every $\omega\in\Omega$, $(M_t(\omega))_{t\ge0}$ is a (positive) martingale with respect to $(\cG_t)_{t\ge0}$. In particular, $E_0[M_t(\omega)] = 1$ for every $t\ge0$.

We are now ready to prove the existence of the tilted free energy.

\begin{theorem}\label{thm:tiltedfe}
	Assume that $V$ satisfies \eqref{ass:wlog}, \eqref{ass:vh} and \eqref{ass:vhol}. Then, the limit in \eqref{eq:tiltedfe} exists for every $\beta>0$, $\theta\in\mathbb{R}$ and $\mathbb{P}$-a.e.\ $\omega$. Moreover, $\Lambda_\beta(\theta) = \beta$ or $\Lambda_\beta(\theta) = \lambda_o(\beta,\theta) > \beta$.
\end{theorem}

\begin{proof}
	If $\Lambda_\beta^U(\theta) = \beta$, then we are done by Lemma \ref{lem:unifLB}. (This is the case, e.g., when $\theta = 0$. See Proposition \ref{prop:list} below.) Otherwise, for every $t>0$, $a>0$ and $\mathbb{P}$-a.e.\ $\omega$,
	\begin{align*}
	E_0\left[e^{\beta\int_{0}^tV(T_{X_s}\omega)ds + \theta X_t}\one_{\{|X_t|\le at\}}\right] &= E_0\left[e^{\beta\int_{0}^tV(T_{X_s}\omega)ds + \theta X_t + F_{\beta,\theta}(\omega,X_t)}\one_{\{|X_t|\le at\}}\right]e^{o(t)}\\
	&= E_0\left[M_t(\omega)\one_{\{|X_t|\le at\}}\right]e^{\lambda_o(\beta,\theta)t + o(t)}\\
	&= \left(1 - E_0\left[M_t(\omega)\one_{\{|X_t| > at\}}\right]\right)e^{\lambda_o(\beta,\theta)t + o(t)}
	\end{align*}
	by Lemma \ref{lem:sublin}, where the $o(t)$ term depends on $a$ and $\omega$. Note that
	\begin{align*}
	0&\le E_0\left[e^{\beta\int_{0}^tV(T_{X_s}\omega)ds + \theta X_t}\one_{\{|X_t|> at\}}\right] \le e^{t(\beta - a)}E_0\left[e^{(|\theta| + 1)|X_t|}\right] \le 2e^{t(\beta - a + (|\theta| + 1)^2)}\quad\text{and}\\
	0&\le E_0\left[M_t(\omega)\one_{\{|X_t| > at\}}\right] \le E_0\left[e^{\sqrt{2\lambda_o}|X_t|}\one_{\{|X_t| > at\}}\right] \le e^{-at}E_0\left[e^{(\sqrt{2\lambda_o} + 1)|X_t|}\right] \le 2e^{t((\sqrt{2\lambda_o} + 1)^2 - a)}
	\end{align*}
	by \eqref{eq:lzmol}. Choosing $a$ sufficiently large concludes the proof.
\end{proof}

The following proposition (whose proof is deferred to Appendix \ref{app:prop}) lists some properties of the tilted free energy.

\begin{proposition}\label{prop:list}
	Assume that $V$ satisfies the conditions in Theorem~\ref{thm:tiltedfe}. Then, the following properties are true.
	\begin{itemize}
		\item [(a)] $\Lambda_\beta(\theta)$ is increasing in $\beta$, and even and convex in $\theta$.
		\item [(b)] $\max\left\{\beta,\theta^2/2\right\} \le \Lambda_\beta(\theta) \le \beta + \theta^2/2$. 
		\item [(c)] $\{\theta\in\mathbb{R}:\,\Lambda_\beta(\theta) = \beta\}$ is a symmetric and closed interval with nonempty interior.
		\item [(d)] The function $\theta\mapsto\Lambda_\beta(\theta)$ is $C^1$ on the complement of $\{\theta\in\mathbb{R}:\,\Lambda_\beta(\theta) = \beta\}$.
	\end{itemize}
\end{proposition}

Proposition \ref{prop:list}(d) is included here as an application of the implicit formula we provide for $\Lambda_\beta(\theta)$ in Theorem \ref{thm:tiltedfe}, and it is not used anywhere in the paper. Note that this proposition does not answer the question of whether the function $\theta\mapsto\Lambda_\beta(\theta)$ is differentiable at the endpoints of the interval $\{\theta\in\mathbb{R}:\,\Lambda_\beta(\theta) = \beta\}$. A negative answer to this question is given in \cite[Appendix D]{YZ17} in the discrete setting.

\section{Upper bounds}\label{sec:UB}

Our next goal is to prove that the limit 
\begin{equation}\label{eq:liman}
\overline H_{\beta,c}(\theta) = \lim_{t\to\infty}\inf_{\alpha\in\mathcal{P}_c}\frac1{t}\log E_0\left[e^{\beta \int_{0}^t V(T_{X_s^\alpha}\omega)ds + \theta X_t^\alpha}\right]
\end{equation}
exists for every $\beta>0$, $c>0$, $\theta\in\mathbb{R}$ and $\mathbb{P}$-a.e.\ $\omega$.
To this end, we define
\begin{equation}\label{al:daif2}
\begin{aligned}
\overline H_{\beta,c}^L(\theta) &= \liminf_{t\to\infty}\inf_{\alpha\in\mathcal{P}_c}\frac1{t}\log E_0\left[e^{\beta \int_{0}^t V(T_{X_s^\alpha}\omega)ds + \theta X_t^\alpha}\right]\\
\text{and}\hspace{15mm}  &\hspace{70mm} \\
\overline H_{\beta,c}^U(\theta) &= \limsup_{t\to\infty}\inf_{\alpha\in\mathcal{P}_c}\frac1{t}\log E_0\left[e^{\beta \int_{0}^t V(T_{X_s^\alpha}\omega)ds + \theta X_t^\alpha}\right].
\end{aligned}
\end{equation}
In this section, we give upper bounds for $\overline H_{\beta,c}^U(\theta)$ by considering specific (families of) policies. In Section \ref{sec:LB}, we obtain matching lower bounds for $\overline H_{\beta,c}^L(\theta)$ and thereby infer that the specific policies considered in Section \ref{sec:UB} are in fact asymptotically optimal. We assume throughout these two sections that $V$ satisfies \eqref{ass:wlog}, \eqref{ass:vh} and \eqref{ass:vhol}.

\subsection{Bang-bang policies}\label{subsec:BB}

For every $c>0$, let
\[{\cP}_c^{BB} = \{\alpha=(\alpha_s)_{s\ge 0}: \ \alpha\ \text{is $\pm c$-valued and $\mathcal {G}_s$-progressively measurable}\},\]
the set of admissible bang-bang policies. For every $\alpha\in\mathcal{P}_c^{BB}$, the It\^o integral $\int_0^t \alpha_s dX_s$ defines a martingale with respect to $(\cG_t)_{t\ge0}$. Since $\langle \int_0^\cdot \alpha_s dX_s\rangle_t = c^2t$, Novikov's sufficient condition for Girsanov's theorem (see \cite[Section 3.5D]{KS91}) is easily satisfied and
\[E_0\left[e^{\beta \int_{0}^t V(T_{X_s^\alpha}\omega)ds + \theta X_t^\alpha}\right] = E_0\left[e^{\beta \int_{0}^t V(T_{X_s}\omega)ds + \int_{0}^t \alpha_sdX_s + \theta X_t}\right]e^{-\frac1{2}c^2t}.\]
Therefore,
\begin{equation}\label{al:daifbang4}
\overline H_{\beta,c}^U(\theta) \le \limsup_{t\to\infty}\inf_{\alpha\in\mathcal{P}_c^{BB}}\frac1{t}\log E_0\left[e^{\beta \int_{0}^t V(T_{X_s}\omega)ds + \int_{0}^t \alpha_sdX_s + \theta X_t}\right] - \frac1{2}c^2.
\end{equation}

\subsection{General upper bound}

Substituting the bang-bang policies $\aleft\equiv -c$ and $\aright\equiv c$ into the RHS of \eqref{al:daifbang4}, we readily get the following bound:
\begin{equation}\label{eq:CUB1}
\overline H_{\beta,c}^U(\theta) \le \min\{\Lambda_\beta(\theta - c), \Lambda_\beta(\theta + c)\} - \frac1{2}c^2.
\end{equation}

\subsection{Upper bound when $|\theta|\le c$}

For every $h\in(0,1)$, $y>0$ and $\mathbb{P}$-a.e.\ $\omega$, assumption \eqref{ass:vh} and the Birkhoff ergodic theorem ensure the existence of an $h$-valley of length $2y$ centered at some $x_*\in\mathbb{R}$. Consider the policy
$\alpha^{(x_*,h,y)}$ defined as
\[\alpha_s^{(x_*,h,y)} = - c\,\sign\left(X_s^{\alpha^{(x_*,h,y)}}-x_*\right).\]
Here, $X_s^{\alpha^{(x_*,h,y)}}$ denotes the unique strong solution of \eqref{eq:SDEmiz} with $\alpha_s = \alpha_s^{(x_*,h,y)}$. For one-dimensional stochastic differential equations with discontinuous drift and nondegenerate diffusion coefficients, strong existence and uniqueness follow from Zvonkin's work \cite{Zvo74} (see \cite[Chapter 5, Section 28]{RW00b}). Therefore, $\alpha^{(x_*,h,y)}\in \mathcal{P}_c^{BB}$. (We take $\sign(0) = 1$ as a convention to get a bang-bang policy.)

Assume without loss of generality that $x_* = 0$. (Starting the BM at $x_*$ instead of the origin does not change the rate of growth of the expectations below.) If $|\theta| \le c$, then
\begin{equation}\label{eq:lahm}
\begin{aligned}
E_0\left[e^{\beta \int_{0}^t V(T_{X_s}\omega)ds + \int_{0}^t \alpha_s^{(0,h,y)}dX_s + \theta X_t}\right] &= E_0\left[e^{\beta \int_{0}^t V(T_{X_s}\omega)ds - c\int_{0}^t \sign(X_s)dX_s + \theta X_t}\right]\\
= E_0\left[e^{\beta \int_{0}^t V(T_{X_s}\omega)ds + c \ell(t,0) - c|X_t| + \theta X_t}\right] &\le E_0\left[e^{\beta \int_{0}^t \one_{\{|X_s|\ge y\}}ds + c \ell(t,0)}\right]e^{\beta ht}
\end{aligned}
\end{equation}
by Tanaka's formula, where $\ell(t,0)$ is the local time of BM at $0$ up to time $t$. 

For every $t'\ge0$, let \[\sigma(t') = \inf\{t\ge0:\, L(t,[-y,y])\ge t'\},\]
where $L(t,[-y,y]) = \int_0^t \one_{\{|X_s|\le y\}}ds$. Observe that $\tilde X_{t'} := |X_{\sigma(t')}|$ is a doubly-reflected BM on $[0,y]$ and denote its local time at $0$ up to time $t'$ by $\tilde\ell^y(t',0)$.
It follows from \cite[Proposition 3.2]{FKZ15} that the limit
\[J_y(c) := \lim_{t'\to\infty}\frac1{t'}\log E_0\left[e^{c\tilde\ell^y(t',0)}\right]\]
exists and satisfies
\begin{equation}\label{eq:nassi}
\lim_{y\to\infty}J_y(c) = \frac1{2}c^2.
\end{equation}
With this notation,
\begin{align*}
E_0\left[e^{\beta \int_{0}^t \one_{\{|X_s|\ge y\}}ds + c \ell(t,0)}\right] &= \sum_{m=0}^{[t]}E_0\left[e^{\beta \int_{0}^t \one_{\{|X_s|\ge y\}}ds + c \ell(t,0)}\one_{\{m \le L(t,[-y,y]) < m +1\}}\right]\\
&\le \sum_{m=0}^{[t]}E_0\left[e^{c \ell(t,0)}\one_{\{m \le L(t,[-y,y]) < m +1\}}\right]e^{\beta(t - m)}\\
&\le \sum_{m=0}^{[t]} E_0\left[e^{c \tilde\ell^y(m+1,0)}\right]e^{\beta(t - m)}\\
&= \sum_{m=0}^{[t]}e^{m J_y(c) + \beta(t - m) + o(m)} = e^{t\max\{\beta,J_y(c)\} + o(t)}.
\end{align*}
Using \eqref{eq:nassi}, we deduce that
\[\limsup_{y\to\infty}\limsup_{t\to\infty}\frac1{t}\log E_0\left[e^{\beta \int_{0}^t \one_{\{|X_s|\ge y\}}ds + c \ell(t,0)}\right] \le \max\left\{\beta,\frac1{2}c^2\right\}.\]
Finally, recalling \eqref{al:daifbang4}, \eqref{eq:lahm} and taking $h\to0$, we obtain the following bound:
\begin{equation}\label{eq:CUB2}
\overline H_{\beta,c}^U(\theta) \le \max\left\{\beta,\frac1{2}c^2\right\} - \frac1{2}c^2 = \left(\beta - \frac1{2}c^2\right)^+.
\end{equation}

%
%

\section{Lower bounds and the effective Hamiltonian}\label{sec:LB}

In this section, we continue assuming that $c>0$ and $V$ satisfies \eqref{ass:wlog}, \eqref{ass:vh} and \eqref{ass:vhol}.

\subsection{Uniform lower bound}

For every $\alpha\in\mathcal{P}_c$,
\begin{align}
E_0\left[e^{\beta \int_{0}^t V(T_{X_s^\alpha}\omega)ds + \theta X_t^\alpha}\right] &= E_0\left[e^{\beta \int_{0}^t V(T_{X_s}\omega)ds + \int_0^t (\theta + \alpha_s) dX_s - \frac1{2}\int_0^t\alpha_s^2ds}\right]\nonumber\\
&\ge E_0\left[e^{\beta \int_{0}^t V(T_{X_s}\omega)ds + \int_0^t (\theta + \alpha_s) dX_s}\right]e^{-\frac1{2}c^2t},\label{eq:duzgunb}
\end{align}
where the equality follows from Girsanov's theorem as in Section \ref{subsec:BB}. Let $Y_t = \int_0^t (\theta + \alpha_s) dX_s$. For every $\xi\in\mathbb{R}$,
\begin{align*}
E_0\left[e^{\xi Y_t}\right] &= E_0\left[e^{\int_0^t \xi(\theta + \alpha_s) dX_s - \frac1{2}\int_0^t \xi^2(\theta + \alpha_s)^2ds + \frac1{2}\int_0^t \xi^2(\theta + \alpha_s)^2ds}\right]\\
&\le E_0\left[e^{\int_0^t \xi(\theta + \alpha_s) dX_s - \frac1{2}\int_0^t \xi^2(\theta + \alpha_s)^2ds}\right]e^{\frac1{2}\xi^2(|\theta| + c)^2t} = e^{\frac1{2}\xi^2(|\theta| + c)^2t}.
\end{align*}
Applying the exponential Chebyshev inequality and optimizing over $\xi$, we see that
\[P_0(Y_t \le -bt) \le e^{-\frac{b^2t}{2(|\theta| + c)^2}}\]
for every $b>0$.

As we argued in the proof of the uniform lower bound when $c=0$ (Lemma \ref{lem:unifLB}), for every $h\in(0,1)$, $y>0$ and $\mathbb{P}$-a.e.\ $\omega$, there exists an $h$-hill of the form $[x^* - y,x^* + y]$ for some $x^*\in\mathbb{R}$. Using the strong Markov property and Lemma \ref{lem:eig}, we get
\[\lim_{t\to\infty}\frac1{t}\log P_0\left(\{\tau_{x^*} \le 1\}\cap\{|X_s- x^*|<y\ \text{for every}\ s\in[\tau_{x^*},t]\}\right) = - \frac{\pi^2}{8y^2}.\]
Therefore,
\[\lim_{t\to\infty}\frac1{t}\log P_0\left(\{\tau_{x^*} \le 1\}\cap\{|X_s- x^*|<y\ \text{for every}\ s\in[\tau_{x^*},t]\}\cap\{Y_t \le -bt\}^c\right) = - \frac{\pi^2}{8y^2}\]
for every $y>0$ sufficiently large (depending on $b$) so that $\frac{\pi^2}{8y^2} < \frac{b^2}{2(|\theta| + c)^2}$. Restricting the expectation on the right-hand side of \eqref{eq:duzgunb} on this intersection of events gives
\begin{align*}
&\liminf_{t\to\infty}\frac1{t}\log E_0\left[e^{\beta \int_{0}^t V(T_{X_s}\omega)ds + \int_0^t (\theta + \alpha_s) dX_s}\right]\\
\ge & \liminf_{t\to\infty}\frac1{t}\log E_0\left[e^{\beta\int_0^t V(T_{X_s}\omega)ds + Y_t}\one_{\{\tau_{x^*} \le 1\}\cap\{|X_s- x^*|<y\ \text{for every}\ s\in[\tau_{x^*},t]\}\cap\{Y_t \le -bt\}^c}\right]\\
\ge & \beta h - \frac{\pi^2}{8y^2} - b.
\end{align*}
Sending $h\to 1$, $y\to\infty$ and finally $b\to0$, we get the following bound:
\begin{equation}\label{eq:CLB1}
\overline H_{\beta,c}^L(\theta) \ge \beta - \frac1{2}c^2.
\end{equation}

\subsection{Lower bound when $\theta \ge 0$ and $\Lambda_\beta(\theta - c) > \beta$}\label{subsec:LBLB}

We start with recording a lower bound for the derivative of the relevant corrector from Section \ref{subsec:cor}.

\begin{lemma}\label{lem:aciktimm}
    If 
	\begin{itemize}
		\item[(i)] $\theta > c$ and $\Lambda_\beta(\theta - c) > \beta$ or
		\item[(ii)] $0 < \theta < c$ and $\beta < \Lambda_\beta(\theta - c) \le c^2/2$,
	\end{itemize}
	then $\theta + (F_{\beta,\theta-c})'(\omega,x) \ge 0$ for every $\omega\in\Omega$ and $x\in\mathbb{R}$.
\end{lemma}

\begin{proof}
	This is an immediate corollary of Lemma \ref{lem:simlaz} now that we know from Theorem \ref{thm:tiltedfe} that $\Lambda_\beta^U(\theta - c) = \Lambda_\beta(\theta - c) = \lambda_o(\beta,\theta-c)$. Indeed, if (i) holds, then
	\[(\theta -c) + (F_{\beta,\theta-c})'(\omega,x) \ge \sqrt{2(\Lambda_\beta(\theta - c) - \beta)} > 0.\]
	Similarly, if (ii) holds, then
	\[(\theta -c) + (F_{\beta,\theta-c})'(\omega,x) \ge -\sqrt{2\Lambda_\beta(\theta - c)} \ge -c.\qedhere\]
\end{proof}

\begin{lemma}\label{lem:LBLB}
	If condition (i) or condition (ii) in Lemma \ref{lem:aciktimm} holds, then
	\begin{equation}\label{eq:CLB2}
	\overline H_{\beta,c}^L(\theta) \ge \Lambda_\beta(\theta - c) - \frac1{2}c^2.
	\end{equation}
\end{lemma}

\begin{proof}
Assume that (i) or (ii) in Lemma \ref{lem:aciktimm} holds. For every $\alpha\in\mathcal{P}_c$, $t\ge0$ and $\omega\in\Omega$, let
\begin{equation}\label{eq:submartan}
M_t^\alpha(\omega) = M_t^\alpha(\omega\,|\,\beta,c,\theta) = e^{\beta \int_{0}^t V(T_{X_s^\alpha}\omega)ds + \theta X_t^\alpha + F_{\beta,\theta-c}(\omega,X_t^\alpha) - \left(\Lambda_\beta(\theta - c) - \frac1{2}c^2\right)t}.
\end{equation}
Observe that $M_0^\alpha(\omega) = e^{\theta x + F_{\beta,\theta-c}(\omega,x)}$ when $X_0^\alpha = x$. Since
\[0 < M_t^\alpha(\omega) < 
e^{c X_t^\alpha + (\theta-c) X_t^\alpha + F_{\beta,\theta-c}(\omega,X_t^\alpha) + \frac1{2}c^2t} \le e^{c X_t^\alpha + \sqrt{2\Lambda_\beta(\theta - c)}|X_t^\alpha| + \frac1{2}c^2t}\]
by Lemma \ref{lem:simlaz}, we have $E_0[(M_t^\alpha(\omega))^2] < \infty$. Applying It\^{o}'s lemma (and using $d\langle X_\cdot^\alpha\rangle_t = dt$), we see that
\begin{align*}
\frac{d M_t^\alpha(\omega)}{M_t^\alpha(\omega)} &= \left(\beta V(T_{X_t^\alpha}\omega) - \Lambda_\beta(\theta - c) + \frac1{2}c^2\right)dt + (\theta + (F_{\beta,\theta-c})'(\omega,X_t^\alpha))(\alpha_t dt + dX_t)\\
&\quad\ + \frac1{2}(\theta + (F_{\beta,\theta-c})'(\omega,X_t^\alpha))^2 dt + \frac1{2}(F_{\beta,\theta-c})''(\omega,X_t^\alpha) dt\\
&= \left(\beta V(T_{X_t^\alpha}\omega) - \Lambda_\beta(\theta - c)\right)dt + (\theta + (F_{\beta,\theta-c})'(\omega,X_t^\alpha))((\alpha_t + c)dt + dX_t)\\
&\quad\ + \frac1{2}(\theta - c + (F_{\beta,\theta-c})'(\omega,X_t^\alpha))^2 dt + \frac1{2}(F_{\beta,\theta-c})''(\omega,X_t^\alpha) dt\\
&= (\theta + (F_{\beta,\theta-c})'(\omega,X_t^\alpha))((\alpha_t + c) dt + dX_t).
\end{align*}
The last equality follows from \eqref{eq:auxODE} with $\theta$ replaced by $\theta-c$. Since $\theta + (F_{\beta,\theta-c})'(\omega,X_t^\alpha) \ge 0$ by Lemma \ref{lem:aciktimm} and $\alpha_t + c \ge 0$, we conclude that $(M_t^\alpha(\omega))_{t\ge0}$ is a (positive) submartingale with respect to $(\cG_t)_{t\ge0}$. In particular, $E_0[M_t^\alpha(\omega)] \ge 1$ for every $t\ge0$.

Finally, for sufficiently large $a>0$,
\begin{align*}
\overline H_{\beta,c}^L(\theta) &= \liminf_{t\to\infty}\inf_{\alpha\in\mathcal{P}_c}\frac1{t}\log E_0\left[e^{\beta \int_{0}^t V(T_{X_s^\alpha}\omega)ds + \theta X_t^\alpha}\right]\\
&= \liminf_{t\to\infty}\inf_{\alpha\in\mathcal{P}_c}\frac1{t}\log E_0\left[e^{\beta \int_{0}^t V(T_{X_s^\alpha}\omega)ds + \theta X_t^\alpha}\one_{\{|X_t^\alpha|\le at\}}\right]\\
&= \liminf_{t\to\infty}\inf_{\alpha\in\mathcal{P}_c}\frac1{t}\log E_0\left[e^{\beta \int_{0}^t V(T_{X_s^\alpha}\omega)ds + \theta X_t^\alpha + F_{\beta,\theta-c}(\omega,X_t^\alpha)}\one_{\{|X_t^\alpha|\le at\}}\right]\\
&= \liminf_{t\to\infty}\inf_{\alpha\in\mathcal{P}_c}\frac1{t}\log E_0\left[e^{\beta \int_{0}^t V(T_{X_s^\alpha}\omega)ds + \theta X_t^\alpha + F_{\beta,\theta-c}(\omega,X_t^\alpha)}\right]\\
&= \liminf_{t\to\infty}\inf_{\alpha\in\mathcal{P}_c}\frac1{t}\log \left(E_0\left[M_t^\alpha(\omega)\right]e^{\left(\Lambda_\beta(\theta - c) - \frac1{2}c^2\right)t}\right) \ge \Lambda_\beta(\theta - c) - \frac1{2}c^2.
\end{align*}
Here, the second and the fourth equalities are easily justified as in the proof of Theorem \ref{thm:tiltedfe}, and the third equality follows from the sublinearity of the corrector (Lemma \ref{lem:sublin}).
\end{proof}

It remains to consider the cases not covered by Lemma \ref{lem:LBLB}. If $\theta = 0$, then (since $V(\cdot)\ge0$) it is clear from \eqref{al:daif2} that
\begin{equation}\label{eq:CLB3}
\overline H_{\beta,c}^L(0)\ge 0.
\end{equation}
If $0 < \theta < c$ and $\beta < c^2/2 < \Lambda_\beta(\theta - c)$, then it follows from Proposition \ref{prop:list} and the intermediate value theorem that there exists a unique $\bar\theta(\beta,c)\in(\theta,c)$ such that
$$\Lambda_\beta(\bar\theta(\beta,c)-c) = \frac1{2}c^2.$$
By Proposition \ref{prop:list}(a), the map $\beta\mapsto\bar\theta(\beta,c)$ is increasing for $\beta\in(0,c^2/2)$, with $\bar{\theta}(0^+,c) = 0$, 
so there exists a unique $\bar\beta = \bar\beta(\theta,c) \in(0,\beta)$ such that $\bar\theta(\bar\beta,c) = \theta$. With this notation, we get the following bound:
\begin{equation}\label{eq:CLB4}
\overline H_{\beta,c}^L(\theta) \ge \overline H_{\bar\beta,c}^L(\theta) \ge \Lambda_{\bar\beta}(\theta - c) - \frac1{2}c^2 = \Lambda_{\bar\beta}(\bar\theta(\bar\beta,c) - c) - \frac1{2}c^2 = 0.
\end{equation}
Here, the first inequality uses $V(\cdot)\ge0$ and the second inequality follows from \eqref{eq:CLB2} which is now applicable since $\bar\beta < \Lambda_{\bar\beta}(\theta - c) = c^2/2$.

\subsection{The effective Hamiltonian}\label{subsec:EffH}

We are ready to check that the lower bounds we have obtained in this section match the upper bounds from Section \ref{sec:UB}.

\begin{theorem}\label{thm:control}
	Assume that $V$ satisfies \eqref{ass:wlog}, \eqref{ass:vh} and \eqref{ass:vhol}. Then, the limit in \eqref{eq:liman} exists for every $\beta>0$, $c>0$, $\theta\in\mathbb{R}$ and $\mathbb{P}$-a.e.\ $\omega$. Moreover, there are two qualitatively distinct control regimes which are characterized by the comparison of $\beta$ and $c^2/2$ (see Figure \ref{matrixfigure}).
	\begin{itemize}
		\item [(a)] (Weak control) If $\beta \ge c^2/2$, then the effective Hamiltonian is given by
		\begin{equation}\label{eq:weaklimit}
		\overline H_{\beta,c}(\theta) = \begin{cases}
		\beta - \frac1{2}c^2&\ \text{if}\ |\theta| < c,\\
		\Lambda_\beta(|\theta| - c) - \frac1{2}c^2&\ \text{if}\ |\theta| \ge c.
		\end{cases}
		\end{equation}
		\item [(b)] (Strong control) If $\beta < c^2/2$, then there exists a unique $\bar\theta(\beta,c)\in(0,c)$ such that
		$$\Lambda_\beta(\bar\theta(\beta,c)-c) = \frac1{2}c^2,$$
		and the effective Hamiltonian is given by
		\begin{equation}\label{eq:stronglimit}
		\overline H_{\beta,c}(\theta) = \begin{cases}
		0&\ \text{if}\ |\theta| < |\bar\theta(\beta,c)|,\\
		\Lambda_{\beta}(|\theta| - c) - \frac1{2}c^2&\ \text{if}\ |\theta| \ge |\bar\theta(\beta,c)|.
		\end{cases}
		\end{equation}
	\end{itemize}
\end{theorem}


\begin{figure}
	\includegraphics{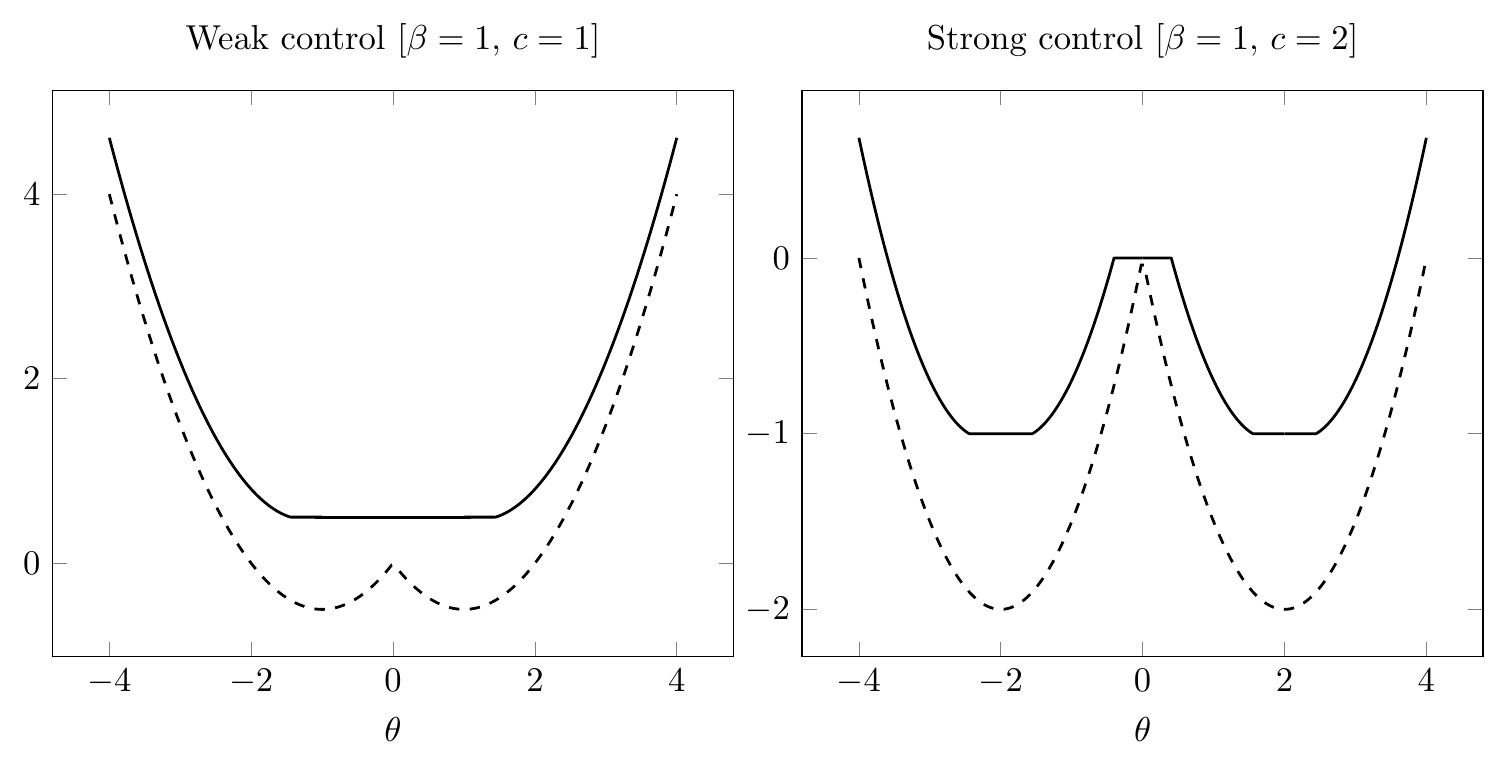}
	\caption{Representative graphs of $\frac1{2}\theta^2 - c|\theta|$ (dashed) and $\overline H_{\beta,c}(\theta)$ (solid) against $\theta$ in each of the two control regimes. There is weak control if and only if $\beta\ge c^2/2$.}
	\label{matrixfigure}
\end{figure}

\begin{proof}
	We prove the existence of the limit in \eqref{eq:liman} separately in each of the two control regimes.
	
	\vspace{3mm}

	(a)\ (Weak control) Assume that $\beta \ge c^2/2$. By symmetry, it suffices to consider $\theta\ge0$.
	\begin{itemize}
		\item If $0\le\theta\le c$, then the bounds \eqref{eq:CUB2} and \eqref{eq:CLB1} match,
		\[\overline H_{\beta,c}^L(\theta) = \overline H_{\beta,c}^U(\theta) = \beta - \frac1{2}c^2,\]
		and taking the infimum in \eqref{eq:liman} over the set $\{\aeL:\,0<h<h_0,\ y>y_0\}$ for any $h_0>0$ and $y_0>0$ does not change the limit.
	
	    \item If $\theta \ge 0$ and $\Lambda_\beta(\theta-c) = \beta$, then the bounds \eqref{eq:CUB1} and \eqref{eq:CLB1} match,
		\[\overline H_{\beta,c}^L(\theta) = \overline H_{\beta,c}^U(\theta) = \beta - \frac1{2}c^2,\]
		and $\aleft$ is asymptotically optimal.
		\item If $\theta > c$ and $\Lambda_\beta(\theta-c) > \beta$, then the bounds \eqref{eq:CUB1} and \eqref{eq:CLB2} match,
		\[\overline H_{\beta,c}^L(\theta) = \overline H_{\beta,c}^U(\theta) = \Lambda_\beta(\theta-c) - \frac1{2}c^2,\]
		and $\aleft$ is asymptotically optimal.
    \end{itemize}
    
    \vspace{3mm}
	
	(b)\ (Strong control) Assume that $\beta < c^2/2$. By symmetry, it suffices to consider $\theta\ge0$. Recall from Section \ref{subsec:LBLB} that there exists a unique $\bar\theta(\beta,c)\in(0,c)$ such that $\Lambda_\beta(\bar\theta(\beta,c)-c) = c^2/2$.
	\begin{itemize}
	\item If $0\le\theta < \bar\theta(\beta,c)$, then the bounds \eqref{eq:CUB2}, \eqref{eq:CLB3} and \eqref{eq:CLB4} match,
	\[\overline H_{\beta,c}^L(\theta) = \overline H_{\beta,c}^U(\theta) = 0,\]
	and taking the infimum in \eqref{eq:liman} over the set $\{\aeL:\,0<h<h_0,\ y>y_0\}$ for any $h_0>0$ and $y_0>0$ does not change the limit.
	\item If $\theta \ge 0$ and $\Lambda_\beta(\theta-c) = \beta$, then the bounds \eqref{eq:CUB1} and \eqref{eq:CLB1} match,
	\[\overline H_{\beta,c}^L(\theta) = \overline H_{\beta,c}^U(\theta) = \beta - \frac1{2}c^2,\]
	and $\aleft$ is asymptotically optimal.
	\item If $\theta \ge \bar\theta(\beta,c)$ and $\Lambda_\beta(\theta-c) > \beta$, then the bounds \eqref{eq:CUB1} and \eqref{eq:CLB2} match,
	\[\overline H_{\beta,c}^L(\theta) = \overline H_{\beta,c}^U(\theta) = \Lambda_\beta(\theta-c) - \frac1{2}c^2,\]
	and $\aleft$ is asymptotically optimal.\qedhere
	\end{itemize}
\end{proof}

\section{Homogenization}\label{sec:homogen}

In this section, we first state (and give references to) existence and uniqueness results for viscosity solutions to the Cauchy problems we have introduced in Section \ref{subsec:main}, then provide the control representation \eqref{eq:protoconrep}-\eqref{eq:controlrep} for $v_\theta$ that we have mentioned there, and finally use \eqref{eq:controlrep} to prove our homogenization results (Theorem \ref{thm:main} and Corollary \ref{cor:main}). We assume throughout the section that $V$ satisfies \eqref{ass:wlog}, \eqref{ass:vreg} and \eqref{ass:vh}. Since we work directly with the control representation rather than the equations themselves, we omit the definition of a viscosity solution and refer the reader to \cite{CIL92,FleSon06}.

\subsection{Existence and uniqueness}\label{subsec:eu}

The following results are well known in very general settings (see, for instance, \cite{CIL92, CL86}). We state them here exactly in the form we need and give precise references for the convenience of the reader.

\begin{lemma}\label{lem:euviscous}
	For every $g\in\text{UC}(\mathbb{R})$, $\epsilon>0$ and $\omega\in\Omega$, the viscous Hamilton-Jacobi equation \eqref{eq} subject to the initial condition $g$ has a unique viscosity solution $u^\epsilon_g$ in $\text{UC}([0,\infty)\times\mathbb{R})$.
\end{lemma}

\begin{proof}
	For every $\omega\in\Omega$, the Hamiltonian $H_{\beta,c}(p,T_x\omega) = \frac12p^2 - c|p| + \beta V(T_x\omega)$ has the following three properties. First,
	\begin{equation}\label{eq:coercive}
	\frac1{4}p^2 - c^2 \le \frac12p^2 - c|p| \le H_{\beta,c}(p,T_x\omega) \le \frac12 p^2 + \beta
	\end{equation}
	for every $(p,x)\in\mathbb{R}\times\mathbb{R}$ since $0 \le (\frac12|p| -c)^2 = \frac1{4}p^2 + c^2 - c|p|$ and $V(T_x\omega)\in [0,1]$ by \ref{ass:wlog}. Second, \begin{equation}\label{eq:ikinci}
	\text{$x\mapsto H_{\beta,c}(p,T_x\omega)$ is Lipschitz continuous}
	\end{equation}
	by \eqref{ass:vreg}. Third,
	\begin{equation}\label{eq:ucuncu}
	|H_{\beta,c}(p,T_x\omega) - H_{\beta,c}(q,T_x\omega)| = \left|\frac12(|p| + |q|) - c\right| ||p| - |q|| \le \left(\frac12(|p| + |q|) + c\right)|p-q|
	\end{equation}
	for every $p,q\in\mathbb{R}$. The desired result follows from \cite[Theorem 2.8]{DK17}.
\end{proof}

\begin{lemma}\label{lem:euinviscid}
	For every $g\in\text{UC}(\mathbb{R})$, the inviscid Hamilton-Jacobi equation \eqref{eqh} subject to the initial condition $g$ has a unique viscosity solution $\overline{u}_g$ in $\text{UC}([0,\infty)\times\mathbb{R})$.
\end{lemma}

\begin{proof}
	It is clear from \eqref{eq:weaklimit}-\eqref{eq:stronglimit} that the Hamiltonian $\overline H_{\beta,c}(p)$ is locally uniformly continuous in $p$. Moreover,
	\[\frac1{4}p^2 - c^2 \le \frac12p^2 - c|p| \le \overline H_{\beta,c}(p) \le \Lambda_\beta(p) \le \frac12 p^2 + \beta\]
	for every $p\in\mathbb{R}$. Here, the first inequality is as in \eqref{eq:coercive}, the second inequality is easy to check by taking $\beta=0$ (see also Figure \ref{matrixfigure}), the third inequality holds since $\alpha\equiv 0\in\cP_c$, and the last inequality is part of Proposition \ref{prop:list}(b). The desired result follows from \cite[Theorem 2.5]{DK17}.
\end{proof}

\subsection{Control representation}\label{subsec:kontem}

For every $\theta\in\mathbb{R}$, $\epsilon>0$, $(t,x)\in[0,\infty)\times\mathbb{R}$ and $\omega\in\Omega$, we will write $u_\theta^\epsilon(t,x,\omega)$ to denote $u_g^\epsilon(t,x,\omega)$ (see Lemma \ref{lem:euviscous}) when $g(x) = \theta x$. Note that
\begin{equation}\label{eq:bezm}
u^\epsilon_\theta(t,x,\omega)= \epsilon u^1_\theta(t/\epsilon,x/\epsilon,\omega)
\end{equation}
by the uniqueness part of Lemma \ref{lem:euviscous}. With this notation, we define
\begin{equation}\label{eq:protoconrep}
\begin{aligned}
v_\theta(t,x,\omega) &= e^{u^1_\theta(t,x,\omega)}\quad\text{and}\\
\tilde v_\theta(t,x,\omega) &= \inf_{\alpha\in\mathcal{P}_c} E_x\left[e^{\beta \int_{0}^t V(T_{X_s^\alpha}\omega)ds + \theta X_t^\alpha}\right].
\end{aligned}
\end{equation}

\begin{lemma}\label{lem:curl}
For every $\theta\in\mathbb{R}$, $(t,x)\in[0,\infty)\times\mathbb{R}$ and $\omega\in\Omega$,
\begin{equation}\label{eq:controlrep}
v_\theta(t,x,\omega) = \tilde v_\theta(t,x,\omega).
\end{equation}
\end{lemma}

\begin{proof}
$\tilde v_\theta(t,x,\omega)$ is the value function of a finite horizon, risk-sensitive stochastic optimal control problem with running payoff function $x\mapsto\beta V(T_x\omega)$ and terminal payoff function $x\mapsto e^{\theta x}$. See \cite[Chapter 6]{FleSon06} for background. As we now show, $\tilde v_\theta$ is a viscosity solution of the associated Bellman equation
\begin{equation}\label{eq:Bellman}
\begin{aligned}
\frac{\partial v}{\partial t} &= \frac1{2}\frac{\partial^2 v}{\partial x^2} + \inf_{|a|\le c}\left(a\frac{\partial v}{\partial x}\right) + \beta V(T_x\omega)v\\
&= \frac1{2}\frac{\partial^2 v}{\partial x^2} - c\left|\frac{\partial v}{\partial x}\right| + \beta V(T_x\omega)v,\quad (t,x)\in(0,\infty)\times\mathbb{R}.
\end{aligned}
\end{equation}
Indeed, when $\theta = 0$ (i.e.\ the terminal payoff function is identically equal to $1$), the last statement follows directly from \cite[Chapter 6, Theorem 8.1]{FleSon06} which is applicable by our assumptions \eqref{ass:wlog} and \eqref{ass:vreg}. On the other hand, when $\theta\ne0$, we can absorb the terminal payoff into the running payoff as follows. For every $\alpha\in\mathcal{P}_c$,
\[E_x\left[e^{\beta \int_{0}^t V(T_{X_s^\alpha}\omega)ds + \theta X_t^\alpha}\right] = E_x\left[e^{\int_{0}^t [\beta V(T_{X_s^{\theta+\alpha}}\omega) + \theta\alpha_s + \frac1{2}\theta^2]ds}\right]e^{\theta x}\]
by Girsanov's theorem. Here, $(\theta + \alpha)_s = \theta + \alpha_s$. Therefore,
\begin{align*}
\tilde v_\theta(t,x,\omega) &= \inf_{\alpha\in\mathcal{P}_c} E_x\left[e^{\int_{0}^t [\beta V(T_{X_s^{\theta + \alpha}}\omega) + \theta\alpha_s + \frac1{2}\theta^2]ds}\right]e^{\theta x}\\
&= \inf_{\alpha\in\mathcal{P}[\theta-c,\theta+c]} E_x\left[e^{\int_{0}^t [\beta V(T_{X_s^\alpha}\omega) + \theta\alpha_s - \frac1{2}\theta^2]ds}\right]e^{\theta x},
\end{align*}
where
\[{\cP}[a,b]=\{\alpha=(\alpha_s)_{s\ge 0}: \ \alpha\ \text{is $[a,b]$-valued and $\mathcal {G}_s$-progressively measurable}\}.\]
Now \cite[Chapter 6, Theorem 8.1]{FleSon06} is applicable and $\tilde v_\theta(t,x,\omega)e^{-\theta x}$ is a viscosity solution of
\begin{equation}\label{eq:davkur}
\frac{\partial v}{\partial t} = \frac1{2}\frac{\partial^2 v}{\partial x^2} + \inf_{|a-\theta|\le c}\left(a\frac{\partial v}{\partial x} + a\theta v\right) + \left(\beta V(T_x\omega) - \frac12 \theta^2\right)v.
\end{equation}
Expressing the derivatives of $\tilde v_\theta(t,x,\omega)$ in terms of those of $\tilde v_\theta(t,x,\omega) e^{-\theta x}$, we deduce that $\tilde v_\theta$ is a viscosity solution of \eqref{eq:Bellman}.

Applying the inverse Hopf-Cole transformation to \eqref{eq:davkur} (see \cite[Chapter 6, Corollary 8.1]{FleSon06}) and arranging the terms, we see that $\log \tilde v_\theta(t,x,\omega) - \theta x$ is a viscosity solution of
\begin{equation}\label{eq:elmadiyen}
\frac{\d u}{\d t}=\frac12\frac{\d^2 u}{\d x^2}+H_{\beta,c}\left(\theta + \frac{\d u}{\d x},T_x\omega\right),\quad (t,x)\in(0,\infty)\times\mathbb{R}.
\end{equation}
Hence, $\log \tilde v_\theta$ is a viscosity solution of \eqref{eq} with $\epsilon = 1$ and $\log \tilde v_\theta(0,x,\omega) = \theta x$. By Lemma \ref{lem:lip} (see below) and the uniqueness part of Lemma \ref{lem:euviscous}, we deduce that $u_\theta^1 = \log \tilde v_\theta$. Consequently, $v_\theta = e^{u_\theta^1} = \tilde v_\theta$.
\end{proof}

\begin{lemma}\label{lem:lip}
	There exists a constant $\kappa$ depending only on $\beta$, $c$, $\theta$ and the Lipschitz constant of $x\mapsto V(T_x\omega)$ such that
	\[|\log \tilde v_\theta(s,x,\omega) - \log \tilde v_\theta(t,y,\omega)| \le \kappa (|s-t| + |x-y|)\] for all $s,t\ge0$, $x,y\in\mathbb{R}$ and $\omega\in\Omega$.
\end{lemma}

\begin{proof}
	Recall from the proof of Lemma \ref{lem:curl} that $\log \tilde v_\theta(t,x,\omega) - \theta x$ is a viscosity solution of \eqref{eq:elmadiyen} and $\log \tilde v_\theta(0,x,\omega) - \theta x = 0$. It follows from \cite[Theorem 3.2]{D16} that $\log \tilde v_\theta(t,x,\omega) - \theta x$ is Lipschitz continuous in $t\in[0,\infty)$ and $x\in\mathbb{R}$ with a Lipschitz constant $\kappa'$ that depends only on $\beta,c,\theta$ and the Lipschitz constant of $x\mapsto V(T_x\omega)$ (which is finite by \eqref{ass:vreg}). This implies the desired result.
\end{proof}

\begin{remark}
	If $x\mapsto V(T_x\omega)$ is in $C_b^2(\mathbb{R})$, then $u_\theta^\epsilon$ is in fact a classical solution of the Cauchy problem \eqref{eq}-\eqref{ic} with $g(x) = \theta x$. See \cite[Chapter 6, Remark 8.1]{FleSon06}.
\end{remark}

\subsection{Proofs of the homogenization results}

We are finally ready to combine our results on the effective Hamiltonian (see Section \ref{subsec:tfe} for $c=0$ and Section \ref{subsec:EffH} for $c>0$) with the control representation of $v_\theta$ (see Section \ref{subsec:kontem}). 

\begin{proof}[Proof of Theorem \ref{thm:main}]
	Recall from the proof of Lemma \ref{lem:euviscous} that $H_{\beta,c}$ satisfies \eqref{eq:coercive}, \eqref{eq:ikinci} and \eqref{eq:ucuncu}.
	For every $\theta\in\mathbb{R}$ and $\mathbb{P}$-a.e.\ $\omega$,
    \[\lim_{\epsilon\to0}u_\theta^\epsilon(1,0,\omega) = \lim_{\epsilon\to0}\epsilon u^1_\theta(1/\epsilon,0,\omega) = \lim_{t\to\infty}\inf_{\alpha\in\mathcal{P}_c}\frac1{t}\log E_0\left[e^{\beta \int_{0}^t V(T_{X_s^\alpha}\omega)ds + \theta X_t^\alpha}\right] = \overline H_{\beta,c}(\theta)\]
    by \eqref{eq:bezm} and Lemma \ref{lem:curl}. The last equality is shown in Theorem \ref{thm:tiltedfe} (resp.\ Theorem \ref{thm:control}) for $c=0$ (resp.\ $c>0$).
    The desired result follows from \cite[Lemma 4.1]{DK17} where the remaining condition (4.2) there holds by Lemma \ref{lem:lip}.
\end{proof}

%
%



\begin{proof}[Proof of Corollary \ref{cor:main}]
	Recall from the proof of Lemma \ref{lem:euviscous} that $H_{\beta,c}$ satisfies \eqref{eq:coercive}, \eqref{eq:ikinci} and \eqref{eq:ucuncu}. The desired result follows from \cite[Theorem 3.1]{DK17} where the remaining condition (L) there holds by Lemma \ref{lem:lip}.
\end{proof}

\appendices

\section{Some properties of the tilted free energy}\label{app:prop}

\begin{proof}[Proof of Proposition \ref{prop:list}]
		(a) These three properties follow from $V(\cdot)\ge0$, the symmetry of the law of BM and a standard application of H\"older's inequality, respectively.
		
		(b) Since $V(\cdot)\in [0,1]$,
		\[\frac1{2}\theta^2 = \frac1{t}\log E_0\left[e^{\theta X_t}\right] \le \frac1{t}\log E_0\left[e^{\beta\int_{0}^tV(T_{X_s}\omega)ds + \theta X_t}\right] \le \frac1{t}\log E_0\left[e^{\beta t + \theta X_t}\right] = \beta + \frac1{2}\theta^2.\]
		Sending $t\to\infty$ and recalling from Theorem \ref{thm:tiltedfe} that $\Lambda_\beta(\theta) \ge \beta$, we obtain the desired bounds.
		
		(c) For every $h\in(0,1)$,
		\begin{align*}
		v_\beta^\beta(\omega,0;1) &= E_0\left[e^{\beta\int_{0}^{\tau_1}[V(T_{X_s}\omega) -1]ds}\right]\\
		&\le E_0\left[e^{\beta(h-1)L(\tau_1,[0,1])}\right]\one_{\{[0,1]\ \text{is an $h$-valley}\}} + \one_{\{[0,1]\ \text{is not an $h$-valley}\}}.
		\end{align*}
		Here, $L(t,[a,b]) = \int_0^t \one_{\{X_s\in[a,b]\}}ds$ denotes the occupation time of BM on the interval $[a,b]$ up to time $t$. Since $P_0(L(\tau_1,[0,1]) > 0) = 1$, we have $E_0\left[e^{\beta(h-1)L(\tau_1,[0,1])}\right] < 1$. Therefore, $\mathbb{E}[\log v_\beta^\beta(\cdot,0;1)] < 0$ by \eqref{ass:vh}. Take any sufficiently small $\theta\ge0$ so that \begin{equation}\label{eq:klst}
		\mathbb{E}[\log v_\beta^\beta(\cdot,0;1)] + \theta < 0.
		\end{equation}
		Recall from the proof of Lemma \ref{lem:dombr} (with $\lambda = \beta$) that
		\begin{equation}
		e^{t(\Lambda_\beta(\theta) - \beta) + o(t)} = E_0\left[e^{\beta\int_0^tV(T_{X_s}\omega)ds + \theta X_t}\right]e^{ - t\beta} \le \sum_{k=0}^\infty e^{\theta + \sum_{j=0}^{k-1}[\log v_\beta^\beta(T_j\omega,0;1) + \theta]}.\nonumber
		\end{equation}
		The RHS is a convergent series since the exponent in the kth term grows linearly in $k$ (with a negative sign) by \eqref{eq:klst} and the Birkhoff ergodic theorem. We deduce that $\Lambda_\beta(\theta)\le\beta$ and conclude by appealing to parts (a) and (b).
		
		(d) Recall from \eqref{eq:defF} that
		\[F_{\beta,\theta}^\lambda(\omega,1) = -\log v_\beta^\lambda(\omega,0;1) - \theta = -\log E_0\left[e^{\beta\int_{0}^{\tau_1}V(T_{X_s}\omega)ds - \lambda\tau_1}\right] - \theta\]
		for every $\omega\in\Omega$, $\theta>0$ and $\lambda\ge\beta$. Since $V(\cdot)\in [0,1]$, it follows from the dominated convergence theorem (DCT) that the function $\lambda\mapsto  F_{\beta,\theta}^\lambda(\omega,1)$ is differentiable for $\lambda>\beta$. By a second application of the DCT, we deduce that the function $\lambda\mapsto  \mathbb{E}[F_{\beta,\theta}^\lambda(\cdot,1)]$ is differentiable and
		\[\frac{\partial}{\partial\lambda}\mathbb{E}[F_{\beta,\theta}^\lambda(\cdot,1)] = \mathbb{E}\left[\frac{E_0\left[\tau_1e^{\beta\int_0^{\tau_1}V(T_{X_s}\omega)ds - \lambda\tau_1}\right]}{E_0\left[e^{\beta\int_0^{\tau_1}V(T_{X_s}\omega)ds - \lambda\tau_1}\right]}\right]>0.\]
		Using the DCT for a third time, we see that $\lambda\mapsto  \mathbb{E}[F_{\beta,\theta}^\lambda(\cdot,1)]$ is continuously differentiable.
		The function $\theta\mapsto\mathbb{E}[F_{\beta,\theta}^\lambda(\cdot,1)]$ is linear, and hence continuously differentiable, too. 
		Recall from Lemma \ref{lem:minzer} and Theorem \ref{thm:tiltedfe} that
		\[0 = \mathbb{E}[F_{\beta,\theta}(\cdot,1)] = \mathbb{E}[F_{\beta,\theta}^{\lambda_o(\beta,\theta)}(\cdot,1)] = \mathbb{E}[F_{\beta,\theta}^{\Lambda_\beta(\theta)}(\cdot,1)]\]
		for $\theta>0$ and $\Lambda_\beta(\theta) > \beta$.
		Thus, by the implicit function theorem, the function $\theta\mapsto\Lambda_\beta(\theta)$ is continuously differentiable on the set $\{\theta\in\mathbb{R}:\,\theta>0\ \text{and}\ \Lambda_\beta(\theta)>\beta\}$. Since $\Lambda_\beta(-\theta) = \Lambda_\beta(\theta)$ and $\Lambda_\beta(0) = \beta$ by parts (a) and (c), this concludes the proof.\qedhere
\end{proof}

\bibliographystyle{alpha}
\bibliography{controlled_BM_ref}

\begin{thebibliography}{GRAS16}

\bibitem[AC15a]{AC15}
Scott~N. Armstrong and Pierre Cardaliaguet.
\newblock Quantitative stochastic homogenization of viscous {H}amilton-{J}acobi
  equations.
\newblock {\em Comm. Partial Differential Equations}, 40(3):540--600, 2015.

\bibitem[AC15b]{AC17}
Scott~N. Armstrong and Pierre Cardaliaguet.
\newblock Stochastic homogenization of quasilinear {H}amilton-{J}acobi
  equations and geometric motions, April 2015.
\newblock Eprint arXiv:math.PR/1504.02045. To appear, J.\ Eur.\ Math.\ Soc.\
  (JEMS).

\bibitem[AS12]{AS12}
Scott~N. Armstrong and Panagiotis~E. Souganidis.
\newblock Stochastic homogenization of {H}amilton-{J}acobi and degenerate
  {B}ellman equations in unbounded environments.
\newblock {\em J. Math. Pures Appl. (9)}, 97(5):460--504, 2012.

\bibitem[AS13]{AS13}
Scott~N. Armstrong and Panagiotis~E. Souganidis.
\newblock Stochastic homogenization of level-set convex {H}amilton-{J}acobi
  equations.
\newblock {\em Int. Math. Res. Not. IMRN}, (15):3420--3449, 2013.

\bibitem[AT14]{AT14}
Scott~N. Armstrong and Hung~V. Tran.
\newblock Stochastic homogenization of viscous {H}amilton-{J}acobi equations
  and applications.
\newblock {\em Anal. PDE}, 7(8):1969--2007, 2014.

\bibitem[ATY15]{ATY15}
Scott~N. Armstrong, Hung~V. Tran, and Yifeng Yu.
\newblock Stochastic homogenization of a nonconvex {H}amilton-{J}acobi
  equation.
\newblock {\em Calc. Var. Partial Differential Equations}, 54(2):1507--1524,
  2015.

\bibitem[ATY16]{ATY16}
Scott~N. Armstrong, Hung~V. Tran, and Yifeng Yu.
\newblock Stochastic homogenization of nonconvex {H}amilton-{J}acobi equations
  in one space dimension.
\newblock {\em J. Differential Equations}, 261(5):2702--2737, 2016.

\bibitem[BL16]{BL16}
Yuri Bakhtin and Liying Li.
\newblock Thermodynamic limit for directed polymers and stationary solutions of
  the {B}urgers equation, July 2016.
\newblock Eprint arXiv:math.PR/1607.04864.

\bibitem[BLP78]{BLP78}
Alain Bensoussan, Jacques-Louis Lions, and George Papanicolaou.
\newblock {\em Asymptotic analysis for periodic structures}, volume~5 of {\em
  Studies in Mathematics and its Applications}.
\newblock North-Holland Publishing Co., Amsterdam-New York, 1978.

\bibitem[CIL92]{CIL92}
Michael~G. Crandall, Hitoshi Ishii, and Pierre-Louis Lions.
\newblock User's guide to viscosity solutions of second order partial
  differential equations.
\newblock {\em Bull. Amer. Math. Soc. (N.S.)}, 27(1):1--67, 1992.

\bibitem[CL86]{CL86}
Michael~G. Crandall and Pierre-Louis Lions.
\newblock On existence and uniqueness of solutions of {H}amilton-{J}acobi
  equations.
\newblock {\em Nonlinear Anal.}, 10(4):353--370, 1986.

\bibitem[CS17]{CS17}
Pierre Cardaliaguet and Panagiotis~E. Souganidis.
\newblock On the existence of correctors for the stochastic homogenization of
  viscous {H}amilton--{J}acobi equations.
\newblock {\em C. R. Math. Acad. Sci. Paris}, 355(7):786--794, 2017.

\bibitem[Dav16]{D16}
Andrea Davini.
\newblock Existence and uniqueness of solutions to parabolic equations with
  superlinear {H}amiltonians, August 2016.
\newblock Eprint arXiv:math.PR/1608.04043.

\bibitem[DH14]{DH14}
Michael Damron and Jack Hanson.
\newblock Busemann functions and infinite geodesics in two-dimensional
  first-passage percolation.
\newblock {\em Comm. Math. Phys.}, 325(3):917--963, 2014.

\bibitem[DK17]{DK17}
Andrea Davini and Elena Kosygina.
\newblock Homogenization of viscous and non-viscous {HJ} equations: a remark
  and an application.
\newblock {\em Calc. Var. Partial Differential Equations}, 56(4):Art. 95, 21,
  2017.

\bibitem[DS09]{DS09}
Andrea Davini and Antonio Siconolfi.
\newblock Exact and approximate correctors for stochastic {H}amiltonians: the
  1-dimensional case.
\newblock {\em Math. Ann.}, 345(4):749--782, 2009.

\bibitem[DS12]{DS12}
Andrea Davini and Antonio Siconolfi.
\newblock Weak {KAM} {T}heory topics in the stationary ergodic setting.
\newblock {\em Calc. Var. Partial Differential Equations}, 44(3-4):319--350,
  2012.

\bibitem[Dyn02]{Dyn02}
E.~B. Dynkin.
\newblock {\em Diffusions, superdiffusions and partial differential equations},
  volume~50 of {\em American Mathematical Society Colloquium Publications}.
\newblock American Mathematical Society, Providence, RI, 2002.

\bibitem[Eva92]{E92}
Lawrence~C. Evans.
\newblock Periodic homogenisation of certain fully nonlinear partial
  differential equations.
\newblock {\em Proc. Roy. Soc. Edinburgh Sect. A}, 120(3-4):245--265, 1992.

\bibitem[FKZ15]{FKZ15}
Martin Forde, Rohini Kumar, and Hongzhong Zhang.
\newblock Large deviations for the boundary local time of doubly reflected
  {B}rownian motion.
\newblock {\em Statist. Probab. Lett.}, 96:262--268, 2015.

\bibitem[FS89]{FlS89}
Wendell~H. Fleming and Panagiotis~E. Souganidis.
\newblock On the existence of value functions of two-player, zero-sum
  stochastic differential games.
\newblock {\em Indiana Univ. Math. J.}, 38(2):293--314, 1989.

\bibitem[FS06]{FleSon06}
Wendell~H. Fleming and H.~Mete Soner.
\newblock {\em Controlled {M}arkov processes and viscosity solutions},
  volume~25 of {\em Stochastic Modelling and Applied Probability}.
\newblock Springer, New York, second edition, 2006.

\bibitem[FS16]{FeS16+}
William~M. Feldman and Panagiotis~E. Souganidis.
\newblock Homogenization and non-homogenization of certain non-convex
  {H}amilton-{J}acobi equations, September 2016.
\newblock Eprint arXiv:math.PR/1609.09410.

\bibitem[Gao16]{G16}
Hongwei Gao.
\newblock Random homogenization of coercive {H}amilton-{J}acobi equations in
  1d.
\newblock {\em Calc. Var. Partial Differential Equations}, 55(2):Paper No. 30,
  39, 2016.

\bibitem[GRAS16]{GRAS16}
Nicos Georgiou, Firas Rassoul-Agha, and Timo Sepp\"al\"ainen.
\newblock Variational formulas and cocycle solutions for directed polymer and
  percolation models.
\newblock {\em Comm. Math. Phys.}, 346(2):741--779, 2016.

\bibitem[JKO94]{JKO94}
V.~V. Jikov, S.~M. Kozlov, and O.~A. Ole{\u\i}nik.
\newblock {\em Homogenization of differential operators and integral
  functionals}.
\newblock Springer-Verlag, Berlin, 1994.
\newblock Translated from the Russian by G. A. Yosifian [G. A. Iosif'yan].

\bibitem[Kri16]{K16}
Arjun Krishnan.
\newblock Variational formula for the time constant of first-passage
  percolation.
\newblock {\em Comm. Pure Appl. Math.}, 69(10):1984--2012, 2016.

\bibitem[KRV06]{KRV06}
Elena Kosygina, Fraydoun Rezakhanlou, and S.~R.~S. Varadhan.
\newblock Stochastic homogenization of {H}amilton-{J}acobi-{B}ellman equations.
\newblock {\em Comm. Pure Appl. Math.}, 59(10):1489--1521, 2006.

\bibitem[KS91]{KS91}
Ioannis Karatzas and Steven~E. Shreve.
\newblock {\em Brownian motion and stochastic calculus}, volume 113 of {\em
  Graduate Texts in Mathematics}.
\newblock Springer-Verlag, New York, second edition, 1991.

\bibitem[LPV87]{LPV}
Pierre-Louis Lions, George Papanicolaou, and S.~R.~S. Varadhan.
\newblock Homogenization of {H}amilton-{J}acobi equation.
\newblock unpublished preprint, circa 1987.

\bibitem[LS03]{LS03}
Pierre-Louis Lions and Panagiotis~E. Souganidis.
\newblock Correctors for the homogenization of {H}amilton-{J}acobi equations in
  the stationary ergodic setting.
\newblock {\em Comm. Pure Appl. Math.}, 56(10):1501--1524, 2003.

\bibitem[LS05]{LS05}
Pierre-Louis Lions and Panagiotis~E. Souganidis.
\newblock Homogenization of ``viscous'' {H}amilton-{J}acobi equations in
  stationary ergodic media.
\newblock {\em Comm. Partial Differential Equations}, 30(1-3):335--375, 2005.

\bibitem[LS10]{LS10}
Pierre-Louis Lions and Panagiotis~E. Souganidis.
\newblock Stochastic homogenization of {H}amilton-{J}acobi and
  ``viscous''-{H}amilton-{J}acobi equations with convex
  nonlinearities---revisited.
\newblock {\em Commun. Math. Sci.}, 8(2):627--637, 2010.

\bibitem[QTY17]{QTY17+}
Jianliang Qian, Hung~V. Tran, and Yifeng Yu.
\newblock {Min-max formulas and other properties of certain classes of
  nonconvex effective {H}amiltonians}, January 2017.
\newblock Eprint arXiv:math.PR/1701.01065.

\bibitem[RT00]{RT00}
Fraydoun Rezakhanlou and James~E. Tarver.
\newblock Homogenization for stochastic {H}amilton-{J}acobi equations.
\newblock {\em Arch. Ration. Mech. Anal.}, 151(4):277--309, 2000.

\bibitem[RW00]{RW00b}
L.~C.~G. Rogers and David Williams.
\newblock {\em Diffusions, {M}arkov processes, and martingales. {V}ol. 2}.
\newblock Cambridge Mathematical Library. Cambridge University Press,
  Cambridge, 2000.
\newblock It\^o calculus, Reprint of the second (1994) edition.

\bibitem[RY99]{RY99}
Daniel Revuz and Marc Yor.
\newblock {\em Continuous martingales and {B}rownian motion}, volume 293 of
  {\em Grundlehren der Mathematischen Wissenschaften [Fundamental Principles of
  Mathematical Sciences]}.
\newblock Springer-Verlag, Berlin, third edition, 1999.

\bibitem[Szn94]{Sz94}
Alain-Sol Sznitman.
\newblock Shape theorem, {L}yapounov exponents, and large deviations for
  {B}rownian motion in a {P}oissonian potential.
\newblock {\em Comm. Pure Appl. Math.}, 47(12):1655--1688, 1994.

\bibitem[Var07]{V07}
S.~R.~S. Varadhan.
\newblock {\em Stochastic processes}, volume~16 of {\em Courant Lecture Notes
  in Mathematics}.
\newblock Courant Institute of Mathematical Sciences, New York; American
  Mathematical Society, Providence, RI, 2007.

\bibitem[Yil11]{Y11}
Atilla Yilmaz.
\newblock Harmonic functions, {$h$}-transform and large deviations for random
  walks in random environments in dimensions four and higher.
\newblock {\em Ann. Probab.}, 39(2):471--506, 2011.

\bibitem[YZ17]{YZ17}
Atilla {Yilmaz} and Ofer {Zeitouni}.
\newblock Nonconvex homogenization for one-dimensional controlled random walks
  in random potential, May 2017.
\newblock Eprint arXiv:math.PR/1705.07613.

\bibitem[{Zil}15]{Z17}
Bruno {Ziliotto}.
\newblock Stochastic homogenization of nonconvex {H}amilton-{J}acobi equations:
  a counterexample, December 2015.
\newblock Eprint arXiv:math.AP/1512.06375. To appear, Comm.\ Pure Appl.\ Math.

\bibitem[Zvo74]{Zvo74}
A.~K. Zvonkin.
\newblock A transformation of the phase space of a diffusion process that will
  remove the drift.
\newblock {\em Mat. Sb. (N.S.)}, 93(135):129--149, 152, 1974.

\end{thebibliography}

\end{document}